\documentclass[11pt, twoside]{article}

\usepackage[a4paper,left=2.5cm,right=2.5cm,top=2.5cm,bottom=2.5cm]{geometry} 
\usepackage{amsmath,amssymb,amsthm,amsfonts,dsfont,mathrsfs}
\usepackage{bbm}
\usepackage{graphicx, fancyhdr}
\usepackage{algorithm}
\usepackage[noend]{algpseudocode}
\usepackage{caption}
\usepackage[margin=0.1cm]{subcaption}
\usepackage{footnote}
\usepackage[dvipsnames]{xcolor}
\usepackage{comment}
\usepackage{array}
\usepackage[shortlabels]{enumitem}

\usepackage{thmtools}
\usepackage{thm-restate}

\usepackage{hyperref} % clickable toc refs
\usepackage[nameinlink]{cleveref}
\let\cref = \Cref

\hypersetup{colorlinks=true,
            citecolor=ForestGreen,
            linkcolor=ForestGreen,    
            urlcolor=Brown}

\usepackage[backend=biber,style=numeric,maxbibnames=99]{biblatex}

\AtEveryBibitem{
    \clearlist{address}
    \clearfield{date}
    \clearfield{eprint}
    \clearfield{isbn}
    \clearfield{issn}
    \clearlist{location}
    \clearfield{month}
    \clearfield{series}
    \clearfield{doi}
    \clearfield{url}

    \ifentrytype{book}{}{
        \clearlist{publisher}
        \clearname{editor}
    }
}

\newtheorem{theorem}{Theorem}[section]
\newtheorem{lemma}[theorem]{Lemma}
\newtheorem{proposition}[theorem]{Proposition}
\newtheorem{corollary}[theorem]{Corollary}

\newtheorem{example}[theorem]{Example}
\newtheorem{remark}[theorem]{Remark}
\newtheorem{definition}[theorem]{Definition}

\newenvironment{enumi}{
    \begin{enumerate}[\upshape 1.,itemsep = 1ex, parsep = 0pt, leftmargin = 7ex]
}{
    \end{enumerate}
}

\newenvironment{enuma}{
    \begin{enumerate}[label = (\alph*), itemsep = 1ex, parsep = 0pt, leftmargin = 7ex]
}{
    \end{enumerate}
}

\let\ga=\alpha \let\gb=\beta  \let\gd=\delta \let\gee=\varepsilon
     \let\gl=\lambda       
 \let\go=\omega  \let\gr=\rho \let\gs=\sigma \let\gt=\tau 
  
 \let\gD=\Delta  \let\gL=\Lambda 
\let\gO=\Omega         \let\gS=\Sigma

\newcommand{\cB}{\mathcal{B}}\newcommand{\cC}{\mathcal{C}}
\newcommand{\cD}{\mathcal{D}}\newcommand{\cF}{\mathcal{F}}
\newcommand{\cH}{\mathcal{H}}\newcommand{\cI}{\mathcal{I}}
\newcommand{\cL}{\mathcal{L}}
\newcommand{\cN}{\mathcal{N}}\newcommand{\cO}{\mathcal{O}}

\newcommand{\cS}{\mathcal{S}}

\newcommand{\bs}[1]{\boldsymbol{#1}}

\newcommand{\fR}{\mathfrak{R}}

\newcommand{\fc}{\mathfrak{c}}

\DeclareMathOperator{\E}{\mathds{E}}
\DeclareMathOperator{\R}{\mathbb{R}}

\DeclareMathOperator{\pr}{\mathds{P}}

\let\var\relax
\DeclareMathOperator{\var}{Var}

\DeclareMathOperator*{\argmin}{argmin}

\DeclareMathOperator{\Id}{Id}
\DeclareMathOperator{\dist}{dist}

\DeclareMathOperator{\spn}{span}
\DeclareMathOperator{\range}{range}
\DeclareMathOperator{\im}{Im}

\renewcommand{\v}[1]{\bs{#1}} 

\renewcommand{\t}[1]{{\mathcal{#1}}} 

\newcommand{\m}[1]{\bs{#1}}

\def\b{\v b}
\def\u{\v u}

\def\x{\v x}
\def\y{\v y}

\def\A{\m A}

\def\D{\m D}
\def\G{\m G}

\def\P{\m P}

\def\S{\m S}
\def\U{\m U}
\def\V{\m V}

\def\X{\m X}

\newcommand{\floor}[1]{\left\lfloor #1 \right\rfloor}
\newcommand{\ceil}[1]{\left\lceil #1 \right\rceil}

\let\ip\relax
\newcommand{\ip}[2]{\left\la #1,\, #2 \right\ra}

\newcommand{\norm}[1]{\left\| #1 \right\|}

\newcommand{\cl}[1]{\overline{#1}}

\renewcommand{\leq}{\leqslant} 
\renewcommand{\geq}{\geqslant} 

\let\it = \textit
\let\bf = \textbf

\def\eset{\varnothing}
\let\sset = \subseteq

\let\inter = \cap
\let\Union = \bigcup

\newcommand{\set}[1]{\left\{#1\right\}}

\def\ra{\rangle}
\def\la{\langle}

\let\eps = \varepsilon

\let\l = \ell

\let\wt = \widetilde

\newcommand{\ind}[1]{\mathds{1}_{#1}}
\DeclareMathOperator{\lip}{Lip}

\DeclareMathOperator{\subG}{subG}
\DeclareMathOperator{\sors}{SORS}
\DeclareMathOperator{\ratnn}{RatNN}
\DeclareMathOperator{\ratcnn}{RatCNN}
\DeclareMathOperator{\relu}{ReLU}
\DeclareMathOperator{\diam}{diam}
\def\din{{d_{\text{in}}}}

\def\bg{\overline{G}}
\let\wt = \widetilde
\let\wh = \widehat
\DeclareMathOperator{\erf}{erf}
\DeclareMathOperator{\crsent}{CrossEntropy}
\DeclareMathOperator{\softmax}{softmax}
\DeclareMathOperator{\vol}{Vol}
\DeclareMathOperator{\const}{\fc}
\DeclareMathOperator{\cp}{CP}
\DeclareMathOperator{\vc}{VCdim}
\DeclareMathOperator{\gerr}{GE}

\let\vec\undefined
\DeclareMathOperator{\vec}{vec}
\DeclareMathOperator{\cell}{cell}
\DeclareMathOperator{\st}{St}

\newcommand{\pd}[2]{\frac{\partial #1}{\partial #2}}

\def\T{\mathbb{T}}

\newtheorem{conjecture}[theorem]{Conjecture}

\date{}

\title{Covering Number of Real Algebraic Varieties and Beyond: \\ Improved Bounds and Applications}

\author{Yifan Zhang\thanks{Oden Institute, UT Austin, \href{mailto:yf.zhang@utexas.edu}{yf.zhang@utexas.edu}}
\and Joe Kileel\thanks{Department of Mathematics and Oden Institute, UT Austin, \href{mailto:jkileel@math.utexas.edu}{jkileel@math.utexas.edu}}}

\begin{document}
\maketitle

\begin{abstract}
    \noindent
    Covering numbers are a powerful tool used in the development of approximation algorithms, randomized dimension reduction methods, smoothed complexity analysis, and others.
    In this paper we prove upper bounds on the covering number of numerous sets in Euclidean space, namely real algebraic varieties, images of polynomial maps and semialgebraic sets in terms of the number of variables and degrees of the polynomials involved.
    The bounds remarkably improve the best known general bound by \cite{yomdin2004tame}, and our proof is much more straightforward.
    In particular, our result gives new bounds on the volume of the tubular neighborhood of the image of a polynomial map and a semialgebraic set, where results for varieties by \cite{lotz2015volume} and \cite{basu2022hausdorff} are not directly applicable.
    We illustrate the power of the result on three computational applications.
    Firstly, we derive a near-optimal bound on the covering number of tensors with low canonical polyadic (CP) rank, quantifying their approximation properties and filling in an important missing piece of theory for tensor dimension reduction and reconstruction.
    Secondly, we prove a bound on dimensionality reduction of images of  polynomial maps via randomized sketching, which has direct applications to large scale polynomial optimization.
    Finally, we deduce generalization error bounds for deep neural networks with rational or ReLU activation functions, improving or matching the best known results in the machine learning literature while helping to quantify the impact of architecture choice on generalization error.
\end{abstract}

\noindent
\bf{keywords:} covering number, real algebraic variety, semialgebraic set, low rank tensor, sketching, polynomial optimization, generalization error, 
\hfill\\

\section{Introduction}

Given a bounded set $V \subseteq \R^N$ in Euclidean space, an $\eps$ net of $V$ is a subset $N \subseteq V$ so that for every $\x\in V$, there exists $\y\in N$ such that $\norm{\x - \y} \leq \gee$.
The minimal cardinality of an $\eps$ net is the \it{covering number} of $V$, denoted as $\cN(V, \gee)$.
The covering number measures the complexity of $V$, in the sense that one only needs $\cN(V, \gee)$ points, or $\log\cN(V, \gee)$ bits of information, to approximate an arbitrary point in the set to a Euclidean distance of $\gee$.
An upper bound on $\cN(V, \gee)$ is thus a powerful tool, widely used in approximation theory and signal reconstruction \cite{siegel2022sharp,rudelson2008sparse}, dimension reduction \cite{woodruff2014sketching,vershynin2018high,martinsson2020randomized}, learning theory and statistical applications \cite{long2019generalization, he2023differentially,le2018persistence}, topological data analysis \cite{niyogi2008finding,di2022sampling}, etc. 
A bound on $\cN(V, \gee)$ also leads to a bound on the volume of the tubular neighborhood $\T(V, \gee) := \Union_{\x \in V} B(\x, \gee)$.
For $V$ the set of bad (or ill-posed) inputs to a given algorithm, $\T(V, \gee)$ can be used in smoothed complexity analysis of the algorithm \cite{basu2022hausdorff,burgisser2013condition}.
Furthermore, in applications the logarithm of the covering number $\log\cN(V, \gee)$, also called the \it{metric entropy} of $V$, is frequently the key quantity to be bounded,
and it has deep connections to other parts of mathematics \cite{ehler2018metric,friedland2007entropy}.

For many sets in $\R^n$ such as the unit sphere in a linear subspace or other smooth submanifolds, the covering number has been extensively studied \cite{vershynin2018high,federer1959curvature,iwen2022fast,szarek1997metric,pajor1998metric,yap2013stable}.
Other works studied covering numbers in infinite dimensional spaces, such as function spaces, under appropriate metrics \cite{tikhomirov1993varepsilon,zhang2002covering,guntuboyina2012covering,siegel2022sharp,bartlett2017spectrally,le2018persistence}.
However, to the best of our knowledge, there is limited work on the covering number of polynomially defined sets, i.e., images of real polynomial maps, algebraic varieties and other semialgebraic sets.
Nonetheless, algebraic sets appear in many applications.
For example, the set of low rank matrices can be viewed as an algebraic variety or the image of the polynomial map defined via  low rank factorization \cite{levin2023finding}, and its covering number is important to compressed sensing algorithms \cite{iwen2021modewise}.
For higher dimensional arrays, the analogs are the sets of low rank tensors and tensor networks \cite{landsberg2013equations}. 
In high dimensional settings as in tensor-based algorithms, dimension reduction and approximation are critical, and thus it is important to have a good bound on the covering number of these algebraic sets.
Although algebraic sets can be viewed as manifolds after the removal of a measure zero subset, it is challenging to apply existing tools to bound the covering number, as they require estimating geometric descriptors such as the volume and the reach of these manifolds \cite{iwen2022fast}, which are often very difficult to access. 

Equipped with good upper bounds on the covering number of algebraic sets, one can analyze nonlinear computational problems such as polynomial optimization and dimension reduction of tensor networks. In addition, the universal approximation power of polynomial maps allows the theory to be applied to derive covering number bounds on many compact (non-algebraic) sets, e.g., as arising from deep neural networks.  We refer the readers to \cref{sec:applic} for more details on possible computational applications.

\subsection{Our Contributions}

We prove a unified bound on the covering number $\cN(V, \gee)$ for all \it{regular sets} $V$ restricted to a compact domain, defined informally as follows:

\begin{definition}\label{def:reg_set_informal}
    We say a set $V \sset \R^N$ is a $(K, n)$ \textup{regular set} if
    \begin{enumi}
        \item For almost all affine planes $L$ of codimension $n' \leq n$ in $\R^N$, $V \inter L$ has at most $K$ path connected components. 
        \item For almost all affine planes $L$ of codimension $n' \geq n+1$ in $\R^N$, $V \inter L$ is empty.
    \end{enumi}
\end{definition}
\noindent See \cref{def:reg_set} for a formal version of \cref{def:reg_set_informal}.  
Classic theory in algebraic geometry shows that real algebraic varieties, images of polynomial maps, and other semialgebraic sets are indeed regular. 
For those sets, $n$ is the intrinsic dimension and $K$ is controlled by a result of early work by Petrovskii and Oleinik \cite{oleinik1951estimates,petrovskii1949topology} and later independent works of Milnor and Thom \cite{milnor1964betti,thom1965homologie}; see \cref{prop:milnor-thom} and \cref{lem:polynom_regularity}.

We are mainly interested in the regime of $n\ll N$. 
This setting is typically the case in  applications where the set of interest $V \sset \R^N$ is parametrized by $n$ variables; it is also a necessity for compression to be possible in order to survive the curse of dimensionality.\footnote{That said, there are applications where $n \approx N$, e.g., smoothed analysis for semi-definite programs \cite{cifuentes2022polynomial}.  However, sets of low codimension are not the main theme of this paper.}
Our main result (\cref{thm:general_cover_number}) states that for a $(K, n)$ regular set $V$ in the cube $[-1, 1]^N$, the covering number is bounded by
\begin{equation}\label{eq:main-bound}
    \log\cN(V, \gee) \leq n \log(1/\gee) + \log K + C n \log N,
\end{equation}
for some absolute constant $C$.\footnote{The choice $C = 3$ is valid unless $N$ is trivially small.}
In particular, if $V$ is the intersection of the unit ball $\cB^N$ with the image of a polynomial map whose coordinate functions have degree at most $d$, then \cref{lem:polynom_regularity} guarantees $\log K = \cO(n \log d)$.
The bound \eqref{eq:main-bound} has an optimal dependence on $\gee$ as $\gee \rightarrow 0$.
We also provide a conjecture and a heuristic justification stating the dependence on $\log N$ seems inevitable given the structure in \cref{def:reg_set_informal}. 
See the discussion at the end of \cref{sec:bound}.

Our main result \eqref{eq:main-bound} only requires the dimension of $V$,
and to compute $K$, by the theorem of Petrovskii-Oleinik-Milnor-Thom, it is enough to have an upper bound on the degree of the defining polynomials when $V$ is a variety or semialgebraic set, or on the degree of the coordinate functions of the polynomial map when $V$ is a polynomial image.
The new bound \eqref{eq:main-bound} remarkably improves the best known general bound on $\cN(V, \gee)$ in \cite{yomdin2004tame}.
In addition, there is no need to acquire a bound on the volume or the reach of $V$, in contrast to existing results that take a different approach using manifolds (e.g., \cite{iwen2022fast}).
Despite recent advances in algorithmically computing bounds on the reach for smooth algebraic varieties \cite{di2023computing}, such methods would require nontrivial, and possibly unaffordable computations in large dimensions. Furthermore, the covering number is a global feature of the set (like volume), and ideally it should not be expressed in terms of a worst-case local descriptor (like the reach).

As a direct consequence of \cref{thm:general_cover_number},  we obtain a new bound on the volume of the tubular neighborhoods of polynomial images (\cref{cor:tubvol_image}) and general semialgebraic sets.

The other contributions of this paper are to three application domains, chosen to illustrate the power of \cref{thm:general_cover_number} in computation.
Firstly, we deduce a bound on the covering number of tensors of low CP rank (\cref{sec:cp}).
It quantifies the fact that a Gaussian tensor is almost never close to low rank, and hence there is a need of an appropriate data generation model for low rank tensor modeling to make sense.
Next, in \cref{sec:sketching} we prove a bound for sketching (i.e., dimension reduction on) the image of a polynomial using random matrices. 
The techniques developed there can also be applied to sketch a general regular set.
As a downstream task, we show how to compress a nonlinear polynomial optimization problem (and some derived variants of it) using a randomized sketch operator, almost as efficiently as when compressing a linear least squares problem (up to a logarithmic factor).
Thirdly, in \cref{sec:gen_err} we leverage the approximation power of polynomial and rational maps to deliver bounds on the generalization error of various deep neural networks -- fully connected or convolutional networks with rational or ReLU activations.  
The bounds in \cref{sec:gen_err} either improve or match the best known results in the machine learning literature, and they give guidance on architecture and activation choice from a generalization error perspective.

\subsection{Comparison to Prior Art}

To our knowledge, the best existing bounds on $\cN(V, \gee)$ for regular sets (\Cref{def:reg_set_informal}) are developed in
\cite{ivanov1975variations,vitushkin1957relation,yomdin2004tame}.
These works studied the covering number of a set $V \sset \R^N$ using its \textit{variations}, defined as
\begin{equation} \label{eq:variation}
    \var_i(V) = c(N, i) \int \var_0(V\inter P) \,dP,
\end{equation}
where $i \leq N$, $c(N, i)$ is some normalization constant, and $\var_0(\cdot)$ denotes the number of connected components of a set. The integral in \eqref{eq:variation} is over all affine subspaces $P$ in $\R^N$ with codimension $i$ with respect to an appropriate measure (see the beginning of \cref{sec:main_result}).
Obviously for bounded regular sets, the variations are easily controlled.
\cite{ivanov1975variations,vitushkin1957relation,yomdin2004tame} derived a bound on the covering number of any set $V$ given bounds on the variations.
In particular, \cite[Proposition 5.8]{yomdin2004tame} showed that for a
semialgebraic set $V \sset \R^N$ of intrinsic dimension at $n$, the covering number is bounded by
\begin{equation*}
    \log\cN(V, \gee) \leq n\log(1/\gee) + \log C(N) + D(V),
\end{equation*}
where $C(N)$ depends on the ambient dimension $N$ and $D(V)$ depends on the sum of the degrees of the defining polynomials and the intrinsic dimension $n$.
The point to make is that the term $\log C(N)$, while not explicitly specified in \cite{yomdin2004tame}, is of order $\gO(N \log N)$ by the argument of \cite{yomdin2004tame}, see \cite[Remark 3.3]{basu2022hausdorff}.
Therefore, our result improves this general bound by reducing the term to $\cO(n \log N)$. 

In the case where $V$ is a semialgebraic set, our result greatly reduces the dependence on the degree of the defining polynomials, by depending only on an upper bound for the degrees instead of their the sum (see item 8 in \cref{lem:polynom_regularity}).
Moreover, our argument is remarkably straightforward compared to   \cite{ivanov1975variations,yomdin2004tame}.

When $V$ is the image of the polynomial map $p: \R^n\rightarrow\R^N$ whose coordinate functions have degree at most $d$, to our knowledge we deduce the first bound in this setup.
An alternative to our approach would be to relate $\im(p)$ to an algebraic variety $U$ by taking its Zariski closure. 
There are recent results on bounding the volume of the tubular neighborhood of a real algebraic variety \cite{lotz2015volume,basu2022hausdorff}, which produce a bound on the covering number by a volume argument (see e.g., \cite[Proposition 4.2.12]{vershynin2018high}):
\begin{equation} \label{eq:vol-arg}
    \cN(V, \gee) \lesssim \cN(U, \gee) \leq \frac{\vol(\T(U, \gee/2))}{\vol( (\gee/2)\cB^N)}.
\end{equation}
Denote $D$ the degree bound of the defining polynomials of variety $U$.
After a careful relaxation of equation (3.2) in \cite{basu2022hausdorff},
\footnote{A tighter but easy relaxation is needed here than the authors used to deduce their Theorem~1.1 in \cite{basu2022hausdorff}.
} 
when $n \ll N$ and $\gee \lesssim 1/(Dn)$, the resulting bound from \eqref{eq:vol-arg} is
\begin{equation*}
    \cN(V, \gee) = n\log(1/\gee) + \cO(n \log N + N \log D).
\end{equation*}
By comparison our result improves the last term $\cO(N\log D)$ to $\log K = \cO(n\log d)$, which is significant since $n\ll N$ and $D$ can be exponentially large compared to $d$ \cite{mayr1997some}.
Most importantly, it is almost never easy to find the polynomials defining $U$ (e.g., \cite{chen2019numerical}).
For example, when $V$ is the set of low rank tensors (see \cref{sec:cp}) and $p$ is the map from the low rank factors to the tensor, understanding the equations defining $U$ is an active area of research \cite{landsberg2013equations}.

Lastly, when $V$ is a real algebraic variety defined by the vanishing of polynomials, our bound has the same asymptotic order as the state-of-the-art \cite{basu2022hausdorff}  when $n \ll N$.
We remark that our argument is much simpler than \cite{lotz2015volume,basu2022hausdorff}. 
\cref{tab:comparison} below summarizes all the comparisons above.

\begin{table}[htbp]
    \renewcommand*{\arraystretch}{1.2}
    \begin{tabular}{c|>{\centering\arraybackslash}m{0.40\linewidth}|>{\centering\arraybackslash}m{0.40\linewidth}}
        \hline
        Source & 
        Conditions on $V$ and $\gee$ & 
        Deduced bound on $\log\cN(V, \gee)$ 
        \\ \hline\hline

        \cite{yomdin2004tame} & 
        $U$ real semialgebraic sets with $B$ equality and inequality constraints, $V = U\inter \cB^N$ &
        $n\log(1/\gee) + \cO(N\log d + N\log B + N\log N)$
        \\ \hline

        \cite{yomdin2004tame} & 
        $U$ real varieties, $V = U \inter \cB^N$ &
        $n\log(1/\gee) + \cO(N\log d + N\log N)$
        \\ \hline

        \cite{burgisser2013condition} &   
        $V$ complex projective varieties, $\gee \lesssim 1/\sqrt{n}$ &  
        $2n\log(1/\gee) + \cO(N\log d + n\log N)$
        \\ \hline

        \cite{lotz2015volume} &   
        $U$ real smooth complete intersections, $V = U\inter \cB^N$ &   
        $n\log(1/\gee) + \cO(N\log d + n\log N)$
        \\ \hline

        \cite{basu2022hausdorff} &  
        $U$ real varieties, $V = U \inter \cB^N$ &   
        $n\log(1/\gee) + \cO(N \log d + n\log N)$
        \\ \hline\hline

        this paper & 
        $U$ real semialgebraic sets with $b$ inequality constraints, $V = U \inter \cB^N$ & 
        $n\log(1/\gee) + \cO(N\log d + b + n \log N)$
        \\ \hline

        this paper & 
        $U$ real varieties, $V = U \inter \cB^N$ & 
        $n\log(1/\gee) + \cO(N\log d + n \log N)$
        \\ \hline

        this paper & 
        $U = \im(p)$, $p:\R^n\rightarrow\R^N$ polynomial, $V = U \inter \cB^N$ & 
        $n\log(1/\gee) + \cO(n\log d + n\log N)$
        \\ \hline
    \end{tabular}
    \caption{Comparison of covering number bounds. Recall $\cB^N$ is the unit ball in $\R^N$, and
    $N$, $n$ are respectively the ambient and intrinsic dimensions. The polynomials defining varieties and semialgebraic sets, as well as the coordinate functions of the polynomial map, are all assumed to have a degree at most $d$.
    The covering number bounds are simplified to the regime of $n \ll N$ and $b = \cO(N\log b)$ in the semialgebraic case.}
    \label{tab:comparison}
\end{table}

\subsection{Other Related Works}

General background on covering numbers can be found in \cite{tikhomirov1993varepsilon}.
There are many previous works on bounding the covering number of various sets:
e.g., compact Riemannian manifolds and curves \cite{roeer2013finite,jean2003entropy},
compact smooth manifolds \cite{federer1959curvature,iwen2022fast, yap2013stable} 
and in particular Grassmannian manifolds \cite{szarek1997metric,pajor1998metric}, 
Lipschitz functions with bounded Lipschitz constant over a compact set \cite{gottlieb2016adaptive,von2004distance}, 
kernel functions \cite{le2018persistence}, 
shallow and deeper neural networks \cite{siegel2022sharp,lin2019generalization,bartlett2017spectrally}, 
spaces of convex functions \cite{guntuboyina2012covering}, 
tensors with low Tucker rank \cite{rauhut2017low} and low CP rank tensors sharing the same set of factor matrices \cite{iwen2021lower}, 
nearly critical points of functions \cite{yomdin2005semialgebraic},
etc. 
Some of these sets are regular sets, e.g., the Grassmannian manifold and low rank Tucker tensors.
By using special properties of these sets, e.g., Lipschitz property of the parameterization of low Tucker rank tensors, sometimes it is possible to avoid dependencies of the covering number on the ambient dimension.  However, such arguments are limited to special cases.

Another topic related to this paper is topological data analysis \cite{carlsson2021topological,niyogi2008finding}. 
The goal is to learn many properties of a space $V$ from noisy sample points on $V$.
Recently, topological data analysis has been initiated for algebraic varieties \cite{breiding2018learning}.
In such applications, it is important to obtain an $(\gee, \gd)$ sample $S$ on $V$, which means $V\sset \T(S, \gee)$ and $S\sset \T(V, \gd)$, see  \cite{dufresne2019sampling}. 
Our \cref{thm:general_cover_number} can easily be used to obtain a bound on the optimal size of $S$, assuming that bounds on the degree of defining equations of $V$ and its dimension are available.

\subsection{Notation}

We use bold lowercase letters (e.g., $\v a, \v b$) for vectors, bold uppercase letters (e.g., $\m A, \m B$) for matrices, and calligraphic uppercase letters (e.g., $\t T$) for higher order tensors. 
For matrices, $\norm{\m A}_2$ denotes the $\ell_2\rightarrow\ell_2$ operator norm and $\norm{\m A}_F$ denotes the Frobenius norm.
The closed unit ball in $\R^N$ is $\cB^{N}$. 
The symmetric cube is denoted as $\cC^{N} := [-1, 1]^N$.
For a set $U \subseteq \R^N$,  
\begin{equation*}
    \pi(U) := \set{\x / \norm{\x}: \x \in U, \x \neq 0}
\end{equation*}
denotes the projection of $U \setminus \{0\}$ to the unit sphere.
We also write $\pi(\x) = \x/\norm{\x}$ for vectors $\x \neq 0$.
The image of function $f$ is denoted $\im(f)$, and for a set $U$ we write $f(U) = \im(f\vert_U)$.
For $n \in \mathbb{N}$, let $[n] = \set{1,2,\ldots,n}$.
The covering number $\cN(V, \gee)$ refers to the covering number by Euclidean balls.
When a different ambient metric $d$ is used, we denote the corresponding covering number by $\cN(V, \gee, d)$.
The function $\log(\cdot)$ refers to the natural logarithm with base $e$.
We use symbol $\const$ to denote an absolute constant.
Note that if multiple instances of $\const$ appear in a derivation, the instances need not refer to the same number.

\section{Main Results} \label{sec:main_result}

Our main result in \cref{thm:general_cover_number} bounds the covering number of all $(K, n)$ regular sets (formally defined in \cref{def:reg_set}). 
In particular, it applies to algebraic varieties and images of polynomial maps (see \cref{sec:regsets}).

We start by defining a $(K,n)$ regular set in $\mathbb{R}^N$ formally.
Denote $\bg_N^k$ the affine Grassmannian manifold whose points are all $k$ dimensional affine planes in $\R^N$.
For $A \in \bg_N^k$, we parametrize it by $(\v x, P)$, where $P \cong \R^{k}$ denotes the tangent space of $A$ and $\v x \in P^{\perp} \cong \R^{N-k}$ denotes the vector in the normal plane corresponding to the  point in $A$ closest to $0 \in \R^N$.
We equip $\overline{G}_N^k$ with the uniform measure $\mu_N^k$ invariant to rigid body transformations.
We refer the readers to \cite{schneider2008stochastic,nicolaescu2020lectures} for more discussion on technical details of measure constructions.
Sometimes we drop the dimensions from the notation when they are obvious from the context.

\begin{definition}\label{def:reg_set}
    We say a set $V \sset \R^N$ is a $(K, n)$ \textup{regular set} in $\R^N$ if
    \begin{enumi}
        \item For $\mu$-almost all planes $L$ of codimension $n' \leq n$ in $\R^N$, $V \inter L$ has at most $K$ path connected components.
        \item For $\mu$-almost all planes $L$ of codimension $n' \geq n + 1$ in $\R^N$, $V \inter L$ is empty.
    \end{enumi}
\end{definition}
\noindent 
In \cref{sec:regsets}, we list some useful classes of regular sets.
In many situations, such as when $V$ is a real algebraic variety, $n$ can be taken as the intrinsic dimension of $V$.
The following properties of regular sets will be useful in the discussion of our main result and its proof.

\begin{lemma}[restate = measregset, name = ]\label{lem:meas_regset}
    Let $V$ be $(K, n)$ regular in $\R^N$.
    Then for $\mu$-almost all planes $L$ in $\R^N$ of codimension $m \leq n$, the set
    $L \inter V$ is $(K, n - m)$ regular in $L \cong \R^{N - m}$.
\end{lemma}

\begin{lemma}[restate = kzeroregset, name = ]\label{lem:k0regset}
    Any $(K, 0)$ regular set is a collection of at most $K$ discrete points.
\end{lemma}

The lemmas are proven \cref{app:mainthm}.
It is also convenient to introduce the following definition.

\begin{definition}[generic planes]
    Let $V$ be a $(K, n)$ regular set.
    We say an affine subspace $L$ of codimension $m \leq n$ is \textup{generic to $V$} if 
    $L$ is not in the $\mu$-null set in
    the condition 1 of \cref{def:reg_set}, nor in  
    the $\mu$-null set in \cref{lem:meas_regset}.
    We say an affine subspace $L$ of codimension $m > n$ is \textup{generic to $V$} if $L$ is not in the $\mu$-null set in condition 2 of \cref{def:reg_set}.
    Consequently, planes generic to $V$ constitute a $\mu$-full measure set in the affine Grassmannian.
\end{definition}

In order for a regular set $V$ to have a finite covering number, we require it to be bounded. The following notation is convenient.

\begin{definition}
    For a set $V$ in $\R^N$, we denote $[V]_{N}$ the smallest $r \geq 0$ such that $V$ can be covered by a set that is a rigid body transformation of $r\cdot \cC^N$, where $\cC^N = [-1,1]^N$.
\end{definition}

\subsection{Main Statement on Covering Numbers}\label{sec:bound} 

For all bounded $(K, n)$ regular set $V$, 
the main result below establishes an upper bound on its covering number.

\begin{theorem} \label{thm:general_cover_number}
    Let $V \sset \R^N$ be a $(K, n)$ regular set with $[V]_N\leq t$. Then for any $\eps \in (0, \diam(V)]$,
    \begin{equation}\label{eq:gen_bnd}
        \log \cN(V, \eps) \leq n \log \left(\frac{2 t n N^{3/2}}{\eps}\right) + \log 2K
        =
        n\log(t/\eps) +
        \cO\Bigl(n \log N + \log K\Bigr).
    \end{equation}
\end{theorem}

\begin{remark}
    The constants in \eqref{eq:gen_bnd} can possibly be improved as we paid little attention to optimize them. 
    Note that a $(K, n)$ regular set is also $(K', n')$ regular for $K' \geq K$ and $n' \geq n$.
    This accords well with our bound above, which is increasing in $K$ and $n$.
    Therefore an upper bound on $K$ and $n$ is sufficient to derive a bound on the covering number. 
\end{remark}

The proof of the main theorem is given in \cref{sec:proof}.
Below we discuss the optimality of our bound on $\log \cN(V, \gee)$.
First we note that our bound is asymptotically tight as $\gee \rightarrow 0$.
Let $V$ be a bounded $n$ dimensional smooth algebraic variety.
It is known that $V$ has a Hausdorff dimension $\dim_{\cH}(V) = n$.
Thus, by a classic result, the entropy dimension
\begin{equation*}
    \dim_e(V) := \limsup_{\gee \rightarrow 0} \frac{\log\cN(V, \gee)}{\log(1/\gee)}
\end{equation*}
is lower bounded by its Hausdorff dimension $n$ (see e.g., \cite[Chapter 2]{yomdin2004tame}).
Since $V$ is $(K, n)$ regular for some $K$, we conclude that for some $(K, n)$ regular sets
$\log \cN(V, \gee) = \Omega(n \log(1/\eps))$ as $\gee\rightarrow 0$.
This shows our bound is asymptotically tight.

Another key concern is whether the $N^{3/2}$ dependence in the numerator in \eqref{eq:gen_bnd} is necessary.
A factor of $N^{1/2}$ is contributed by the fact that we are covering a set in the cube $\cC^N$ with $\ell_2$ balls.
Indeed if we bound the covering number of $V$ with respect to the $\ell_{\infty}$ metric using the same argument, then in \eqref{eq:gen_bnd} the term $N^{3/2}$ is reduced to $N^1$.
We conjecture that an $N^{1}$ dependence is necessary.

\begin{conjecture} \label{conj}
    There exists a $(K, n)$ regular set $V$ in $\cB^N$ with covering number $\cN(V, \eps) = \Omega(N^n)$.
\end{conjecture}

To justify this conjecture, consider the intersection of $V$ with coordinate hyperplanes, $V_i = V \inter \set{x_i = 0}$. Then \cref{lem:meas_regset} implies $V_i$ are $(K, n-1)$ regular\footnote{If the coordinate hyperplanes happen to be non-generic, we perturb them by a small amount.} in $\cB^{N-1}$. Then $\cN(V, \eps) \gtrsim \sum_{i=1}^N \cN(V_i,\eps) \asymp N \cdot \cN(V_1,\eps)$ is plausible.  So by induction, we expect $\cN(V, \eps) = \Omega(N^n)$.

\color{black}

Even assuming \Cref{conj} holds, it is an interesting question whether for specific subclasses of regular sets,  e.g., varieties and semialgebraic sets, the $N$ dependence can be improved in the regime $n \ll N$.  We leave this to a future investigation.

Lastly, we note that the dependency on $K$ is natural in \eqref{eq:gen_bnd}.
A discrete set of $K$ elements is a $(K, 0)$ regular set, and it has a covering number of $K$ when $\gee$ is sufficiently small.

\subsection{Proof of Theorem~\ref{thm:general_cover_number}}\label{sec:proof}

Our proof of \Cref{thm:general_cover_number} is based on a \it{slicing argument}.
The idea is to intersect $V$ with a lower dimensional affine subspace $L$, and bound the covering number of $V\inter L$.
Starting from the base case where $L$ has a codimension $n$ and $V\inter L$ is a finite set of points, we inductively decrease the codimension of $L$ and finally conclude a bound on $V$.
The following definitions and lemma are elementary but handy.

\begin{definition}
    [cells]\label{def:cell}
    Let $U$ be a set in $\R^N$ and $\cL = \set{L_i}_{i \in \cI}$ a collection of affine hyperplanes.
    For $\x, \y\in U$, define the equivalence relation on $U \setminus \cup_{L_i \in \mathcal{L}} L_i$ by $\x \sim \y$ if for all $i\in\cI$, $\x$ and $\y$ are on the same side of plane $L_i$.
    We denote $\cell(U, \cL)$ the collection of all such equivalence classes.
    We denote $\diam(\cell(U, \cL))$ the supremum of the diameters of all elements in $\cell(U, \cL)$. 
\end{definition}

\begin{definition}[$\eps$ regular system]\label{def:reg_sys}
    Let $V$ be a bounded regular set in $\R^N$.
    For any $\gee > 0$,
    we say a finite collection $\cL$ of affine hyperplanes is an $\gee$ regular system on $V$ if 
    $\diam(\cell(V, \cL)) \leq \gee$ and each $L\in\cL$ is generic to $V$.
\end{definition}

\begin{lemma}
    \label{lem:reg_sys}
    Fix any $t > 0$, $\gee > 0$, $\gt > 1$.
    Let $V \sset \R^N$ be a regular set with $[V]_N\leq t$.
    Then there exists an $\gee$ regular system $\cL$ on $V$ of size $2 t N^{3/2}\gee^{-1}$, such that for any plane $L\in \cL$, it holds $[L \inter V]_{N-1} \leq \gt t$.
\end{lemma}

\begin{proof}
    By translation, we may assume $V\sset t\cdot \cC^N$, and it suffices to have $[L \inter (t\cdot \cC^N)]_{N-1} \leq \gt t$ in the regular system.
    Let $E = \set{\v e_1,\ldots, \v e_N}$ be the standard basis of $\R^N$.
    For some $\gd > 0$ chosen later,
    put a grid with grid spacing no larger than $((1-\gd) \gee/\sqrt{N})$ on $[-t, t]$, with size at most 
    \begin{equation*}
        M \leq \floor{2t\sqrt{N}(\gee(1 - \gd))^{-1}}.
    \end{equation*}
    Note that we excluded the end points at $\set{-t, t}$ from the grid. 
    Denote the grid points by $\set{\theta_i}_i$.
    For each $i \in [M]$ and $k \in [N]$, denote $L^k_i = \theta_i \v e_k + \v e_k^\perp$.

    Now clearly the planes $\cL = \set{L^k_i: i\in [M], k \in [N]}$ partition $t\cdot \cC^N$, and thus $V$, into small cubes of diameter at most $(1-\gd)\eps$.
    The cardinality of $\cL$ is 
    \begin{equation*}
        |\cL|\leq N\cdot \floor{2t\sqrt{N}((1-\gd)\gee)^{-1}}, 
    \end{equation*}
    and for each $L \in \cL$, $[L \inter (t\cdot \cC^N)]_{N-1} = t$.
    By \cref{lem:meas_regset} and the fact that $t\cdot \cC^N$ is bounded, we can perturb each $L\in \cL$ by a sufficiently small amount to get $L'$ and consequently a perturbed collection $\cL'$, so that (i) each $L'$ is generic to $V$, (ii) $\diam(\cell(\cL', V)) \leq \gee$, and (iii) $[L'\inter (t\cdot \cC^N)]_{N-1} \leq \gt t$. 
    Since $\gd> 0$ is arbitrary, we can take it sufficiently small and conclude that $|\cL'| \leq 2 t N^{3/2}\gee^{-1}$.
    Since (iii) implies $[L'\inter V]_{N-1} \leq \gt t$ for $L'\in\cL'$, $\cL'$ is a valid candidate for the desired $\gee$ regular system.
\end{proof}

Lastly we introduce one more piece of language needed in the main proof.
An illustration of \cref{def:types} can be found in \cref{fig:idea}.

\begin{definition}\label{def:types}
    Let $V\sset \R^N$.
    Let $\cL$ be a collection of affine hyperplanes.
    Let $W$ be a path connected component of $V\setminus\Union_{L\in\cL}L$.
    We call $W$ type-I if $\cl{W}\inter V \inter L = \eset$ for all $L \in \cL$, that is, $\cl{W}\inter V$ is contained in the interior of some $E\in\cell(\R^N, \cL)$.
    Otherwise we call $W$ a type-II component.
\end{definition}

Below use the set illustrated \cref{fig:idea} (a) as a simple example in $\R^2$ to explain the main idea behind the proof.

\begin{example}
    Consider a compact $(K, 1)$ regular set $V$ in $\R^2$ illustrated by the colored curves in \cref{fig:idea} (a) (for this curve, $K = 4$ is sufficient).
    The black grid denotes an $\eps/\sqrt{2}$ regular system $\cL$ so the diameter of the cells are bounded by $\eps$.
    The green circle is a type-I component, and the blue curve is the union of type-II components.

    To cover both curves, we put a single ball for each type-I component to take care of all the green components.
    Then we put a ball at each intersection of the blue curve and the black grid (e.g., the red dots). This will cover the blue curve of type-II components.

    Since there are at most $K$ green components (because there are at most $K$ connected components in $V$ overall) and since each black line has at most $K$ intersections with $V$, the covering number $\cN(V, \eps)$ is bounded as $\cN(V, \eps) \leq K + |\cL| \cdot K$.
    \cref{lem:reg_sys} gives a bound on the number of planes $|\cL|$, and therefore we have a bound on the total number of balls.
\end{example}

\begin{figure}[!h]
    \centering
    \begin{minipage}{0.3\textwidth}
        \hspace*{-18ex}(a)\\
        \centering
        \includegraphics[width = \linewidth]{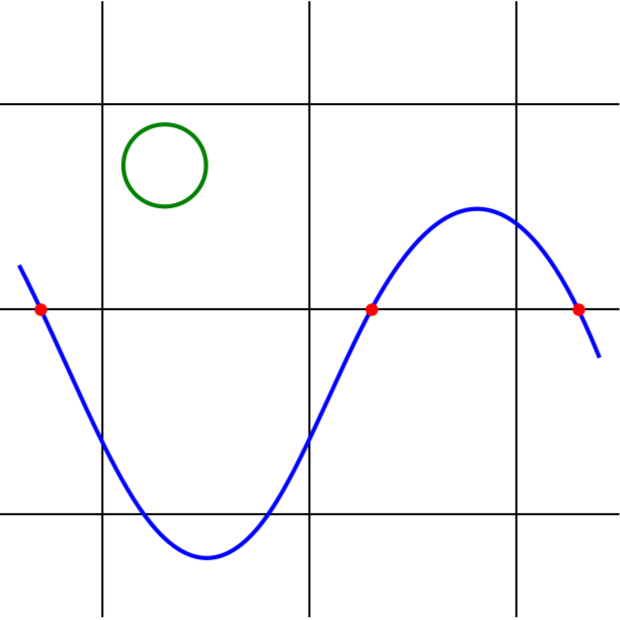}
    \end{minipage}
    \hspace*{0.1\textwidth}
    \begin{minipage}{0.3\textwidth}
        \hspace*{-18ex}(b)\\
        \centering
        \includegraphics[width = \linewidth]{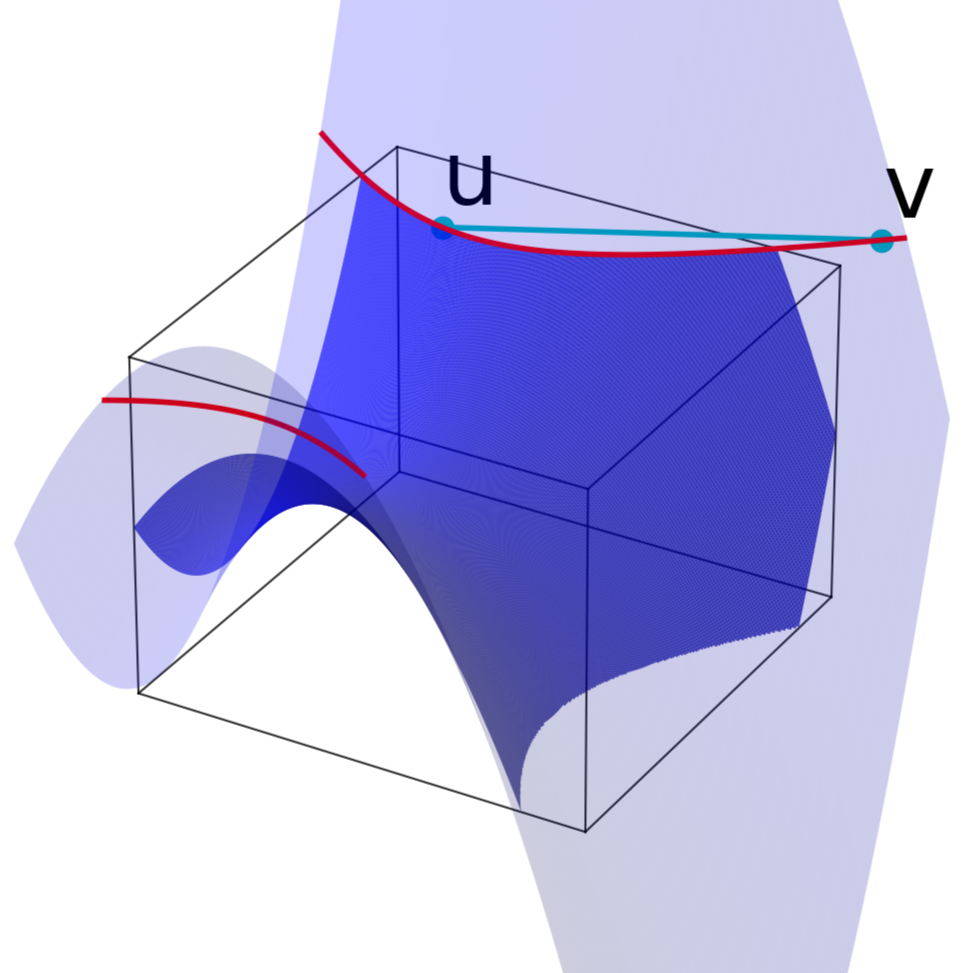}
    \end{minipage}
    \caption{
        (a): An example where $\dim(V) = 1$ in $\R^2$. The black lines represent planes in $\cL$. There are a circle-shaped type-I component colored in green, and a wave-shaped blue curve containing a few type-II components. In (b) we have $\dim(V) = 2$ in $\R^3$ (the blue surface).  The black box denotes a cell of the regular system. We can identify the intersection of the roof of the black cell (a generic plane) and $V$, which is the red curve, as an instance of the curve in (a).}
    \label{fig:idea}
\end{figure}

Going into higher dimensions, each red dot in \cref{fig:idea} (a) is no longer a dot that is covered by a single ball, but rather a lower dimensional regular set (see \cref{fig:idea} (b)).
This allows us to do the induction as the dimension of the intersection grow and get the final result.
The full proof is given below.

\begin{proof}[Proof of \cref{thm:general_cover_number}]
    The bound is obviously valid if $N = 1$ or $n = 0$. Consider $N \geq 2$ and $n \geq 1$.
    Take some $\gd \in (0, 2t]$, which will be specified later.
    Let $L_k$ be an affine plane of dimension $N - n + k$, $0 \leq k \leq n$.
    Denote $V_k = V \inter L_k$.
    We prove by induction that for all $0\leq k \leq n$,
    if $L_k$ is generic and $[V_k]_{N-m} \leq \gt t$ for some $\gt > 1$, then $V_k$ is covered by at most
    \begin{equation} \label{eq:Mk}
        M_k = (2t/\gd)^k \gt^{c_k} K \prod_{j = 1}^k (N - n + j+1)
    \end{equation}
    balls centered in $V_k$ of radius
    \begin{equation} \label{eq:rk}
        r_k = k\sqrt{N}\gd,
    \end{equation}
    where $c_0 = 0$ and $c_k > 0$. The exact values of $c_k$ will not be important.

    The base case $k = 0$ is automatic, since by \cref{lem:meas_regset}, $V_0$ is $(K, 0)$ regular, which is a set of at most $K$ points by \cref{lem:k0regset}.
    Thus, they are covered by $M_0 = K$ balls of radius $r_0 = 0$.

    Next consider $L_k$ being a generic plane having codimension $N - n + k$, $k \geq 1$, and $[L_k \inter V]_{N-n+k} \leq\gt t$.
    Again, $V_k$ is $(K, k)$ regular in $L_k$.
    Let $\cL$ be a $\sqrt{N}\gd$ regular system on $V_k$.
    By \cref{lem:reg_sys}, we can take $\cL$ so that all planes $L\in\cL$ are generic to $V_k$, and
    \begin{equation*}
        |\cL| \leq 2 \gt t \cdot \frac{(N-n+k)^{3/2}}{\sqrt{N}\gd} \,\leq (2t/\delta)\tau \cdot (N-n+k),
    \end{equation*}
    and in addition each plane $L\in \cL$ satisfies $[L \inter V_k]_{N-n+(k-1)} \leq \gt^2 t$.
    Therefore, by the induction hypothesis, for all $L \in \cL$, $L\inter V = L\inter V_k$ is covered by at most
    \begin{equation*}
        M_{k-1} = (2t/\gd)^{k-1} \gt^{2c_{k-1}} K \prod_{j = 1}^{k-1} (N - n + j+1)
    \end{equation*}
    balls of radius $r_{k-1} = (k-1)\sqrt{N}\gd$.

    Let us go back to $V_k = L_k \inter V$.
    By regularity conditions, there are at most $K$ type-I components on $V_k$.
    For each type-I component $W$, we have $\diam(W) \leq \diam(\cell(V, \cL)) \leq \sqrt{N}\gd$.
    Hence, all type-I components are covered by $K$ balls of radius $\sqrt{N} \gd$ with centers in $V_k$.

    Next, take an arbitrary type-II component $W$.
    Take arbitrary $\v u \in \Union_{L \in \cL} \cl{W}\inter V_k \inter L$. 
    Since $W$ has a diameter is at most $\sqrt{N} \gd$, $W \sset B(\v u, \sqrt{N}\gd)$.
    Now we identify $L \in \cL$ that contains $\v u$. 
    By the induction hypothesis,
    there are centers $\set{\v v_i}$, $i\in [M_{k-1}]$ on $L \inter V_k$ so that $\Union_i B(\v v_i, r_{k-1})$ covers $L\inter V_k$.
    In particular, there is some $\v v \in \set{\v v_i}_i$ so that $d(\v u, \v v) \leq r_{k-1}$.
    Thus, we conclude that $W$ is covered by a ball centered at $\v v$ of radius $r_{k} = r_{k-1} + \sqrt{N}\gd$.
    For each $L \in\cL$ and for each center $\v v_i$ in $L$, we put a ball of radius $r_k$ centered at $\v v_i$. 
    These balls then cover all type-II components. 
    Consequently, all type-I and type-II components are covered by at most
    \begin{align*}
        K + |\cL|M_{k-1}
        &\,\leq\,
        K + (2 t/\gd)\gt\cdot (N-n+k) \cdot (2t/\gd)^{k-1} \gt^{2c_{k-1}} K \prod_{j = 1}^{k-1} (N - n + j+1)\\
        &\leq 
        (2 t/\gd) \gt\cdot (N-n+k + 1) \cdot (2t/\gd)^{k-1}  \gt^{2c_{k-1}} K \prod_{j = 1}^{k-1} (N - n + j+1)\\
        &=
        M_k
    \end{align*}
    balls of radius $r_k$.
    For points in $V_k$ that is not in any type-I or type-II component, they are in some $L \in \cL$, and thus are also covered by these balls by the induction hypothesis.
    This completed our induction step.

    Finally we apply the result on $k = n \geq 1$.
    The covering number of $V$ by balls of radius $\gee = n\sqrt{N} \gd$ is at most
    \begin{equation}\label{eq:bound_in_proof}
        M_n = (2t/\gd)^n \gt^{c_n} K \prod_{j = 1}^n (N - n + j+1)
        \leq \left(2t n \frac{N^{3/2}}{\gee}\right)^n \cdot 2K.
    \end{equation}
    In the last inequality, 
    note that for all $1\leq n\leq N$, $N \geq 2$, we have $\prod_{j = 1}^n (N - n + j+1) \leq (3/2) N^n$.
    we choose $\gt$ sufficiently close to $1$ so that $\gt^{c_n} \prod_{j = 1}^n (N - n + j+1) \leq 2N^n$.
    For $\gd > 2t$, $\gee > 2t\sqrt{N} \geq \diam(V)$.
    In this trivial case, $V$ is covered by a single ball of radius $\gee$. 
    The proof is thus complete by taking the logarithm of the right-hand side of \eqref{eq:bound_in_proof}.
\end{proof}

\subsection{Examples of Regular Sets}\label{sec:regsets}

As our next concern, we show in this section that some important sets are indeed regular sets.
Namely, this include images of polynomial and rational maps, as well as algebraic varieties and semialgebraic sets.
We remark that the bound on $K$ for semialgebraic sets therein drastically improves the result in \cite[Corollary 4.9]{yomdin2004tame}.
The main handle to bound $K$ for polynomially defined sets is the following theorem of Petrovskii-Oleinik-Milnor-Thom (POMT).

\begin{proposition}[POMT] \label{prop:milnor-thom}
Let $\set{f_1,\ldots,f_k}$ be a collection of polynomials on $\R^N$ of degree at most $d$. 
Denote $Z(f_i)$ the zero set of $f_i$.
Then the number of path connected components\footnote{Path connectedness is equivalent to connectedness on semialgebraic sets, see e.g. \cite[Proposition 2.5.13,  Theorem 2.4.5]{bochnak2013real}.} of $V = \inter_i Z(f_i)$ is upper bounded by $d(2d-1)^{N-1}$. 
\end{proposition}

We show how to use POMT to control $K$ when $V$ is a polynomial image in \cref{lem:polynom_regularity}.
Then we list more examples of regular sets with explicit bounds on $(K, n)$ in \cref{lem:regvariety} to \cref{lem:regsemi}.
The proofs to these results are similar to the polynomial case, and thus are given in \cref{app:regset}.

\begin{lemma}[polynomial images are regular]\label{lem:polynom_regularity}
    Let $p:\R^n \rightarrow \R^N$ be a polynomial map whose coordinate functions are of degree at most $d$. Then
    \begin{enuma}\setlength\itemsep{0.5em} 
        \item $\im(p)$ is $((2d)^n, n)$ regular.
        \item $\im(p) \inter (c \cdot \cB^N)$ is $((4d)^{n+1}, n)$ regular for all $c > 0$.
        \item $\pi(\im(p))$ is $((4d + 1)^{n+1}, n)$ regular. If $\im(p)$ is a cone, then the second regularity parameters may be changed from $n$ to $n-1$.
    \end{enuma}
\end{lemma}

\begin{proof}
    Let $L = (\u, P)$ be an affine subspace of $\mathbb{R}^N$ of codimension $m$ (see the beginning of \cref{sec:main_result} for this parametrization).
    Let $\v e_1,\ldots,\v e_m$ be an orthonormal basis of $P^\perp$. 
    For a subset $S$ of Euclidean space, let $CC(S)$ denote the number of path connected components of $S$ with respect to the Euclidean topology.
    Note that item 4 and item 5 are immediate from item 8.
    We prove the rest of the assertions.

    For item 1, consider $S = \set{\x: p(\x) \in L}$ in $\R^n$. 
    It is the solution set of a system of degree $d$ polynomial equations:
    \begin{equation*}
        \ip{p(\x)}{\v e_i} = \ip{\v u}{\v e_i},\ \ \ i = 1,\ldots,m.
    \end{equation*}
    Thus by POMT (\cref{prop:milnor-thom}),
    $CC(S) \leq d(2d-1)^n \leq (2d)^n$.  
    The same is true for $\im(p) \cap L$, because $\im(p) \cap L = p(S)$ and $p$ is continuous.
    When $m > n$, the polynomial system is overdetermined and has no solution for Zariski generic $\u$ and $P$. Indeed, see \cite[Theorem~1.1]{dalbec1995introduction} or \cite[Section~3.2.B]{gelfand1994discriminants}, 
    and note $\im(p)$ is contained in a variety of dimension at most $n$ (namely its Zariski closure).
    We conclude that $\im(p)$ is $((2d)^n, n)$ regular.

    For (2), 
    let $w$ be a dummy variable and consider $S = \set{(\x, w) : \|p(\x)\|^2 + w^2 = c^2, \, p(\x) \in L}$ in $\R^n \times \R$.  It is given by the solution set of 
    \begin{gather*}
        \ip{p(\x)}{\v e_i} = \ip{\v u}{\v e_i},\ \ \ i = 1,\ldots,m;\\
        \norm{p(\x)}^2 + w^2 = c^2.
    \end{gather*} 
    The polynomial equations have degree bounded by $2d$.   So POMT implies $CC(S) \leq (4d)^{n+1}$.
    When $m > n$, $S = \eset$ Zarski-generically by \cite[Theorem~1.1]{dalbec1995introduction}.
    Since the map $(\x, w) \mapsto p(\x)$ is continuous and sends $S$ to $\im(p) \cap  (c\cdot \cB^M) \cap L$, we conclude that $\im(p) \cap  (c \cdot \cB^M)$ is $((4d)^{n+1}, n)$ regular.

    For (3),
    consider $S = \set{(\x, w):  \| p(\x) \|^2 w  =1, \, \pi(p(\x)) \in L}$ in $\R^n \times \R$. 
    Then $S$ is the solution set to
    \begin{gather*}
        \ip{p(\x)}{\v e_i}^2 = \ip{\v u}{\v e_i}^2 \norm{p(\v x)}^2,\ \ \ i = 1,\ldots,m;\\
        \norm{p(\x)}^2 w = 1.
    \end{gather*}
    The bottom equation excludes the zero set of $p$, so that $\pi(p(\x))$ is well-defined. 
    Note all the equations have degree at most $2d + 1$.  
    By POMT, $CC(S) \leq (2d+1)(4d + 1)^{n} \leq (4d+1)^{n+1}$.
    Again, we have $S = \eset$ when $m > n$ and $L$ is Zariski generic.
    Since the map $(\x, w)\mapsto \pi(p(\x))$ is continuous on $S$ with image $\pi(\im(p)) \cap L$, we conclude that $\pi(\im(p))$ is $((4d+1)^{n+1}, n)$ regular.
    Finally, if $\im(p)$ is a cone then $\pi(\im(p))$ is contained in a variety of dimension $1$ less.  
    Hence, we can replace $n$ by $n-1$ in the second regularity parameter.
\end{proof}

\begin{lemma}[restate = regvariety, name = varieties are regular]\label{lem:regvariety}
    Let $V \sset \R^N$ be a variety defined by $\{\forall j, f_j = 0 \}$, where $f_j$ are polynomials of degree at most $d$, $d \geq 2$.  Suppose $V$ has dimension at most $n$. 
    Then
    \begin{enuma}\setlength\itemsep{0.5em}
        \item $V$ is $((2d)^N, n)$ regular.
        \item $V \inter (c \cdot \cB^N)$ is $((2d)^{N + 1}, n)$ regular for all $c > 0$.
        \item $\pi(V)$ is $((2d+1)^{N+1}, n)$ regular. If $V$ is a cone, then the second regularity parameters may be changed from $n$ to $n-1$.
    \end{enuma}
\end{lemma}

\begin{lemma}[restate = regrat, name = images of rational maps are regular]\label{lem:regrat}
    Let $f$ be a rational map from an open dense subset of $\R^n$ to $\R^N$ whose coordinate functions are ratios of polynomials of degree at most $d$ (see \cref{def:rational}). Then $\im(f)$ is $((2Nd + 1)^{n+1}, n)$ regular. 
\end{lemma}

\begin{lemma}[restate = regsemi, name = semialgebraic sets are regular]\label{lem:regsemi}
    Let $U \sset \R^N$ be a semialgebraic set defined by $\set{\forall k,~p_k = 0} \inter \set{\forall \ell,~q_\ell \leq 0} \inter \set{\forall t,~r_t < 0}$\footnote{In general, a semialgebraic set is a finite union of such $U$ defined here. It is trivial to show that the union of a $(K_1, n_1)$ regular set and a $(K_2, n_2)$ regular set is $(K_1 + K_2, \max\set{n_1, n_2})$ regular.}, where $p_k,~q_\ell,~r_t$ are polynomials of degree at most $d$. Suppose $U$ has dimension at most $n$ and there are $b$ inequality constraints. Then $U$ is $((2d)^{N}\min\set{(2d)^b, 7^b N, (\const b^2)^N}, n)$ regular\footnote{The term $7^b N$ can be improved with $N$ replaced with the dimension of the variety defined by polynomials $p_k$.}.
\end{lemma}

The above lemmas showed that polynomially defined sets are indeed regular.
However, the class of regular sets is more general.  
For example, it can be shown that trigonometrically-parameterized sets are regular with explicit bounds on $n$ and $K$.  Another example is convex sets in affine subspaces and the relative boundaries of convex sets.

\subsection{Corollaries on Tubular Volumes}

The bound on the covering number gives a bound on the volume of tubular neighborhoods $\T(V, \gee)$ for regular sets $V$, as $\T(V, \gee)$ is covered by the union of radius $2\eps$ balls centered at the points in the $\eps$ covering net of $V$.

\begin{corollary} \label{cor:tubvol_regset}
    Let $V$ be a $(K, n)$ regular set in $\R^N$ with $[V]_N\leq t$.
    Then
    \begin{equation*}
        \vol(\T(V, \gee)) \leq \cN(V, \gee) \cdot (2\gee)^N \vol(\cB^N)
        \leq \vol(\cB^N) 2^N\gee^{N-n} K \cdot (\const t N^{3/2} n)^n.
    \end{equation*}
\end{corollary}

\begin{proof}
    By triangle inequality, for any $\gee$ net $U$ on $V$, we have $\T(V, \gee) \sset \Union_{\x \in U} B(\x, 2\gee)$.
    Substituting in our bound in \cref{thm:general_cover_number} finishes the proof.
\end{proof}

We can then apply this to bound the volume of tubular neighborhoods of polynomial images.
The following version of \cref{cor:tubvol_regset} is specialized to the setup studied in \cite{lotz2015volume,basu2022hausdorff}.

\begin{corollary}\label{cor:tubvol_image}
    Let $V \in \R^N$ be the image of a polynomial map whose coordinate functions have a degree at most $d$.
    Suppose the dimension of $V$ is bounded by $n$.
    Given $\v p \in \R^N$ and $\gs > 0$, let $\x$ be uniformly distributed in $B(\v p, \gs)$.
    Denote $\dist(\x, V)$ the minimal Euclidean distance between $\x$ and $V$.
    Then for any $\gee > 0$,
    \begin{equation*}
        \log \pr\Big(\dist(\x, V) \leq \gee\Big) \leq (N - n)\log (\gee/\gs) + N \log 2 + \const \cdot\Big(n\log d + n\log N \Big).
    \end{equation*}
\end{corollary}

\begin{proof}
    To begin with, we note that
    \begin{equation*}
        \pr\Big(\dist(\x, V) \leq \gee\Big) 
        \leq
        \frac{\vol\big(\T(V \inter B(\v p, \gs + \gee), \gee)\big)}{\vol\big(B(\v p, \gs)\big)}.
    \end{equation*} 
    Indeed, when $\T(V, \gee) \inter B(\v p, \gs)$ is empty, the probability is trivially 0.
    Otherwise, for any $\v y \in \T(V, \gee) \inter B(\v p, \gs)$, $\norm{\v y - \v p} \leq \gs$, and there exists some $\v z \in V$ so that $\norm{\v y - \v z} \leq \gee$.
    Thus, $\v z \in V \inter B(\v p, \gs + \gee)$, and since $\v y \in B(\v z, \gee)$, we have $\T(V, \gee) \inter B(\v p, \gs) \sset \T(V \inter B(\v p, \gs + \gee), \gee)$.
    The probability bound hence follows.

    We bound the numerator using \cref{cor:tubvol_regset}.
    To this end, $V \inter B(\v p, \gs + \gee) $ is a $((4d)^{n+1}, n)$ regular set by \cref{lem:polynom_regularity}. 
    Thus,
    \begin{equation*}
        \log \vol\big(\T(V \inter B(\v p, \gs + \gee), \gee)\big)
        \leq
        \log \vol(\cB^N) + (N-n) \log \gee + n \log \gs + N\log 2 + \const n\log(Nd),
    \end{equation*}
    where we used $n \leq N$ and $\gee \leq \gs$.
    Finally, with $\log \vol(B(\v p, \gs)) = \log(\vol(\cB^N)) + N \log \gs$, we arrive at the claimed bound on the logarithm of the probability.
\end{proof}

Similarly, one can deduce bounds for the volume of the tubular neighborhood for algebraic varieties and semialgebraic sets from \cref{thm:general_cover_number}.
When $V$ is a real algebraic variety, the corresponding bound on the volume has the same scaling as the one in \cite{basu2022hausdorff}. 
To our knowledge, our bound for tubular neighborhoods of general semialgebraic sets is new. 
Note that the semialgebraic set given by $\set{\forall k,~p_k = 0} \inter \set{\forall\l,~q_\ell \leq 0} \inter \set{\forall t,~r_t < 0}$ can have a lower dimension compared to the bigger variety $\set{\forall k,~p_k = 0}$.

\section{Applications}\label{sec:applic}

In this section, we present diverse applications of \cref{thm:general_cover_number}.
The first application is rather straightforward: we bound the covering number of tensors of low canonical polyadic (CP) rank.
The CP rank of a tensor $\t T$, sometimes simply called the rank, is a generalization of the matrix rank, i.e., the smallest number of rank-1 tensors that sum up to $\t T$ (see more details in \cref{sec:cp}).
Secondly, we study dimensionality reduction for polynomial images (and other regular sets), and how it can accelerate optimization problems (under a general loss).
Thirdly, we derive bounds on generalization errors for deep neural networks with rational and ReLU activations.  
This choice of applications illustrates the power of \cref{thm:general_cover_number} in different domains.

\subsection{Covering Number of Low Rank CP Tensors}\label{sec:cp}

In many tensor problems, low-rank decompositions can remarkably improve the computation and storage efficiency.
One classic example is approximating a general tensor by a tensor of low CP rank, which is a high order extension of low rank matrix approximation (see e.g., \cite{kolda2009tensor,anandkumar2014tensor,7891546} for the importance of low-rank tensor decompositions).
A covering number bound on the set of low rank tensors characterizes the approximation power.
Moreover, it gives guarantees on the accuracy when randomized dimension reduction techniques are applied to algorithms finding the low rank approximations (see \cref{sec:sketching}).
Among commonly-used low rank structures on tensors, covering number bounds for the tensor train (TT), hierarchical Tucker, and Tucker tensors format have been derived \cite{rauhut2017low,iwen2021modewise}.
However, to the best of our knowledge, there is no such bound available for low rank CP tensors (defined below), despite CP tensors being one of the most popular low rank tensor formats. 
We show how to use \cref{thm:general_cover_number} to fill this gap.
In fact, \cref{thm:general_cover_number} is general enough to give covering number bounds for arbitrary tensor networks \cite{orus2014practical}.

Let $\t T \in \R^{n_1\times \ldots\times n_d}$ be a $d$th order tensor.
The tensor is said to be of \textit{CP rank at most $r$} if there exist vectors $\v{a}^{(j)}_i \in \R^{n_i}$  ($i = 1,\ldots, d$ and $j=1, \ldots, r$) such that
\begin{equation*}
    \t T  =  \sum_{j = 1}^r \v a_1^{(j)}\otimes \ldots\otimes \v a_d^{(j)}.
\end{equation*}
Here $\v a\otimes\v b$ denotes the tensor product.
We call $\m A_i = \big{(}\v a^{(1)}_i|\ldots|\v a^{(r)}_i\big{)} \in \R^{n_i\times r}$ ($i = 1,\ldots,d$) the factor matrices, and denote $\t T = \cp(\m A_1,\ldots,\m A_d)$.
In the matrix case $\t T \in \R^{n_1\times n_2}$ of rank $r$, the formula reads in the familiar way $\t T = \A_1\A_2^\top$.
Our new bound on the covering number of low rank CP tensors is as follows.

\begin{theorem}
    \label{thm:covering_low_cp_rank}
    Let $V \sset \R^{n_1\times\ldots\times n_d}$ be the set of tensors with CP rank at most $r$ where $n_i \geq 2$ for each $i$.
    Denote $\overline{n} = \frac{1}{d} \sum_{i = 1}^d n_i$ as the average dimension.
    Then for any $t > 0$, $\gee \in (0, 2t]$,
    \begin{equation*}
        \log \cN(V \inter t \cdot \cB^{n_1\times \ldots \times n_d}, \gee)
        \leq
        rd \overline{n} \cdot \log(t/\gee)
        + 
        \const r d \overline{n} \cdot \sum_{i = 1}^d \log n_i.
    \end{equation*}
    If in fact $r \leq \min_i n_i$, then for some $c_1, c_2 \geq 1$,
    \begin{equation*}
        \log \cN(V \inter t \cdot \cB^{n_1\times \ldots \times n_d}, \gee)
        \leq 
        rd\overline{n} \log(c_1 d t/\gee) + c_2 d^2r^2 \log r - dr^2\log(c_1),
    \end{equation*}
    Since $V$ is a cone, the bounds above also apply to $\log \cN(\pi(V), \gee)$ with $t = 1$ and $rd\overline{n}$ replaced by $rd\overline{n}-1$.
\end{theorem}

\begin{remark}
    The first bound in \cref{thm:covering_low_cp_rank} applies to any rank $r$, and has an essentially optimal dependence in $\gee$. 
    The second bound in \cref{thm:covering_low_cp_rank} applies only to small $r$, but in this regime it improves the term independent of $\gee$, achieving an optimal linear dependence on $n_i$.
\end{remark}

We need the following technical lemma for the second bound in \cref{thm:covering_low_cp_rank}.

\begin{lemma}[restate = grasstostief, name = ] \label{lem:grass2stief}
    Let $E$ and $F$ be two points on the Grassmannian manifold $G_n^k$. 
    Let $\m P_E$, $\m P_F$ be the orthogonal projections to $E$ and $F$, respectively.
    If $\norm{\m P_E - \m P_F}_2 \leq \gee$, then for any basis $\m Q_E$ of $E$, there is a basis $\m Q_F$ of $F$ such that $\norm{\m Q_E - \m Q_F}_2 \leq \sqrt{2}\gee$.
\end{lemma}
\begin{proof}
See \cref{app:cp}.
\end{proof}

\begin{proof}[proof of \cref{thm:covering_low_cp_rank}]
    A rank at most $r$ tensor $\t T$ can be written as $\t T = \sum_{j = 1}^r \bigotimes_{i = 1}^d \v a^{(j)}_i$ for some vectors $\v a^{(j)}_i \in \R^{n_i}$.
    Thus, $V$ is the image of a polynomial with coordinate function of degree $d$ and $r \sum_i n_i$ variables, and $\diam(V) = 2t$.
    The first bound follows directly from \cref{lem:polynom_regularity} and \cref{thm:general_cover_number} after simplification. 
    When the cone $V$ is projected to the sphere, by \cref{lem:polynom_regularity} we can reduce $rd\overline{n}$ by 1.

    For the second bound, since $V$ is a cone, it is sufficient to prove the claim assuming $t = 1$.
    Take $t = 1$, and thus $\gee \leq 2$.
    We exploit the idea of stabilizing the parametrization via canonicalization (see e.g., \cite{lee2014fundamental,zhang2020stability}).
    Let $\t T = \cp(\m A_1,\ldots,\m A_d)$, with $\m A_i \in \R^{n_i\times r}$ being the factor matrices.
    We can orthogonalize $\m A_i$ by $\m A_i = \m U_i \m R_i$ for some orthonormal basis $\m U_i$.
    Here if $\m A_i$ is not full rank, we put the rank deficiency in the $\m R_i$ factor and keep $\m U_i$ as an orthonormal matrix in $\R^{n_i \times r}$.
    Now $\t T$ as a vector can be written as
    \begin{equation*}
        \vec(\t T) = \Big(\bigotimes_{j = 1}^d \m U_j\Big) \cdot \vec(\cp(\m R_1,\ldots,\m R_d)),
    \end{equation*}
    where $\vec(\cdot)$ flattens a tensor into a vector.
    Denote $\v r = \vec(\cp(\m R_1,\ldots,\m R_d))$.
    By the triangle inequality, for a different rank at most $r$ tensor $\wh{\t T} = \cp(\wh{\m U}_1 \wh{\m R}_1,\ldots, \wh{\m U}_d \wh{\m R}_d)$ of norm at most 1, we have
    \begin{align} \label{eq:err_decomp_cp}
        \|\t T - \wh{\t T}\|_F 
        &\leq \Big\| \Big(\bigotimes_{j = 1}^d \m U_j \Big) \cdot \v r - \Big(\bigotimes_{j = 1}^d {\m U}_j \Big) \cdot \wh{\v r} \Big\| + \Big\| \Big(\bigotimes_{j = 1}^d \m U_j \Big) \cdot \wh{\v r} - \Big(\bigotimes_{j = 1}^d \wh{\m U}_j \Big) \cdot \wh{\v r} \Big\| \nonumber \\
        &\leq  \| \v r - \wh{\v r} \| + \Big\| \Big(\bigotimes_{j = 1}^d \m U_j \Big) - \Big(\bigotimes_{j = 1}^d \wh{\m U}_j \Big) \Big\|_2 \nonumber \\
        &\leq \| \v r - \wh{\v r} \| + \sum_{j=1}^d \| \m U_j - \wh{\m U}_j \|_2
    \end{align}
    where $\hat{\v r} = \vec(\cp(\wh{\m R}_1,\ldots,\wh{\m R}_d))$.

    To construct an $\gee$ net on $V$, we put an $\gee/(2\sqrt{2} d)$ net on each Grassmannian $G_{n_i}^k$, whose size can be bounded by \cite{pajor1998metric,szarek1997metric}, and we use \cref{lem:grass2stief} to bound the error terms $\|\m U_i - \wh{\m U}_i\|_2$.
    Then we put another $\gee/2$ net on the space of all possible $\v r$ vectors, which we denote as $V_R$.
    The size of this net can be bounded using the first bound of this theorem. 

    To the latter end, by our first bound in the statement, 
    \begin{equation} \label{eq:cover_core}
        \log\cN(V_R, \gee/2) \leq dr^2 \log(d/\gee) + c_2 d^2r^2\log r, 
    \end{equation}
    Denote $W_R$ a minimal $\gee/2$ net of $V_R$.
    Next, by \cite{szarek1997metric,pajor1998metric} we have\footnote{For points $E$ and $F$ in $G_{n}^k$, the 2-norm in \eqref{eq:cover_q} is $\norm{E - F}_2 := \norm{\m P_E - \m P_F}_2$, with $\m P_E$ and $\m P_F$ being the projections to $E$ and $F$, respectively.}
    \begin{equation} \label{eq:cover_q}
        \log\cN(G_{n_i}^r, \gee/(2\sqrt{2}d), \norm{\cdot}_2) \leq \left(n_i r - r^2\right) \log(c_1 d/\gee).
    \end{equation}
    Now for each $i \in [d]$, find a minimal $\gee/(2\sqrt{2}d)$ net $W_i$ on $G_{n_i}^r$.
    For each point $E$ in the net, fix an arbitrary orthonormal basis $\m Q_E$ of $E$ to represent $E$.
    To complete the proof, we claim that the tensors in $V$ with $\m U_i$ factors in $\set{\m Q_E: E\in W_i}$ and $\v r$ vectors in $W_R$ form an $\gee$ net on $V$.

    Indeed,
    for any $\t T = \cp(\m U_1\m R_1,\ldots,\m U_d\m R_d)$, look at its $i$th factor $\m U_i\m R_i$.
    Let $F_i = \spn(\m U_i)\in G_{n_i}^k$.
    By the net assumption of $W_i$ and \cref{lem:grass2stief}, we can find some $E_i\in W_i$ and another basis $\m Q_{F_i}$ of $F_i$ so that $\|\m Q_{E_i} - \m Q_{F_i}\|_2 \leq \gee/(2d)$.
    Under the new basis of $F$, we have $\m A_i = \m Q_{F_i} \wt{\m R}_i$ for some new coordinates $\wt{\m R}_i$.
    Repeat this reparametrization process for each $i \in [d]$.
    Let the new $\v r$ vector for $\t T$ be $\tilde{\v r} = \vec(\cp(\wt{\m R}_1,\ldots,\wt{\m R}_d))$.
    We can then find some $\v r \in W_R$ so that $\norm{\v r - \tilde{\v r}}_2 \leq \gee/2$.
    By \eqref{eq:err_decomp_cp}, we get
    \begin{equation*}
        \Big\|\vec(\t T) - \Big(\bigotimes_{i = 1}^d \m Q_{E_i}\Big)\cdot \v r\Big\| \leq d\cdot \gee/(2d) + \gee/2 = \gee. 
    \end{equation*}
    The latter term on the left-hand side is an element of the net, and thus the claim is justified.
    Finally, the size of the $\gee$ net on $V$ is bounded by
    \begin{align*}
        \log\cN(V \inter \cB^{n_1\times \ldots\times n_d}, \gee)
        &\leq
        \log\cN(V_R, \gee/2) + d\log\cN(G_{n_i}^r, \gee/(2\sqrt{2}d), \norm{\cdot}_2)\\
        &=
        rd\overline{n} \log(c_1 d t/\gee) + c_2 d^2r^2 \log r - dr^2\log(c_1).
    \end{align*} 
    The proof is thus complete.
    When the cone $V$ is projected to the sphere, it is equivalent to project the cone $V_R$ to the sphere. 
    Thus, in \eqref{eq:cover_core}, we can replace the first term by $(dr^2-1)\log(d/\gee)$.
    This leads to reducing $rd\overline{n}$ to $rd\overline{n}-1$ in the final bound.
\end{proof}

As a quick illustration of the covering number, we derive a bound for the probability that a random Gaussian tensor is $\gee$-close to a CP rank at most $r$ tensor.

\begin{corollary} \label{cor:cp_approx}
    Let $\t T \in \R^{n_1\times\ldots\times n_d}$ be a random tensor with i.i.d. standard Gaussian entries.
    Denote $\measuredangle(\t T, \t T')$ the angle between $\t T$ and $\t T'$ under the Frobenius inner product.
    Denote $N := \prod_i n_i$.
    Suppose $N\geq 8$.
    Then for any $r \leq \min_i n_i$ and $\gee \leq \pi/6$, we have 
    \begin{equation*}
        \pr\Big(\exists \,\t T' \text{~of rank at most~} r \text{~s.t.~} \measuredangle(\t T, \t T') \leq \gee \Big)
        \leq 
        \frac{\sin(2\gee)^{N-1}}{\gee^{rd\overline{n}-1}} \cdot \frac{(c_1 d)^{rd\overline{n}}\cdot r^{c_2 d^2 r^2}}{\sqrt{N} (c_1)^{dr^2}}
        =
        \cO\Big(\gee^{N - rd\overline{n}}\Big).
    \end{equation*}
    where $c_1,c_2 \geq 1$ are absolute constants.
\end{corollary}

\begin{proof}
    Denote $V$ the cone of tensors of rank at most $r$.
    Since $\pi(\t T)$ is uniformly distributed on the sphere, it suffices to bound the volume of the tubular neighborhood around $\pi(V)$.
    On the sphere, it is convenient to use the geodesic distance $d_g$.
    The covering number of geodesic balls of $\pi(V)$ is then bounded by
    \begin{equation*}
        \cN(\pi(V), \gee, d_g) \leq \cN(\pi(V), 2\sin(\gee/2), \norm{\cdot}_2).
    \end{equation*}
    We then bound the volume of its tubular neighborhood on the sphere by
    \begin{equation*}
        \vol \T(\pi(V), \gee, d_g) \leq \cN(\pi(V), 2\sin(\gee/2),\norm{\cdot}_2) \cdot \vol \cB(2\gee, d_g),
    \end{equation*}
    where $\cB(2\gee, d_g)$ is the geodesic ball of radius $2\gee$.
    By \cite[Exercise 7.9]{boucheron2013concentration} we have the following tight bound for $\vol \cB(2\gee, d_g)$ when $\cos(2\gee) \geq \sqrt{2/N}$, which holds as $\gee \leq \pi/6$ and $N \geq 8$:
    \begin{equation*}
        \frac{\vol \cB(2\gee, d_g)}{\vol \cS^{N-1}} \leq \frac{1}{2\cos(2\gee)\sqrt{N}}\sin(2\gee)^{N-1} \leq \frac{\sin(2\gee)^{N-1}}{\sqrt{N}}.
    \end{equation*}
    Finally, note that the probability of interest is simply $\vol \T(\pi(V), \gee, d_g) /\vol \cS^{N-1}$.
    We use the relaxation $2\sin(\gee/2) \geq \const \gee$ for $\gee \leq \pi/6$, and substitute in our bound for $\cN(\pi(V), \const \gee, \norm{\cdot}_2)$ 
    in \cref{thm:covering_low_cp_rank}.
    This leads to the bound given in the corollary statement. 
\end{proof}

\begin{remark}
    \cref{cor:cp_approx} implies that given a random tensor the probability to obtain a good low rank approximation decays fast as $\gee \rightarrow 0$, especially when $\overline{n}$ is large. 
    As a numerical reference, set $d = 3$, $n_i = \overline{n} = 100$, $r = 30$, and take the constants in the bound to be 3.
    Then for $\gee = \pi/6$ we have $\pr < e^{-38567}$, and for $\gee = \pi/7$ we have $\pr < e^{-139455}$.  
    The conclusion is that low rank tensor decomposition requires special structure in the input tensor, e.g., as in \cite{udell2019big,zhang2022moment}.
\end{remark}

\subsection{Dimension Reduction of Regular Sets via Random Sketching} \label{sec:sketching}

In this section, we use the covering number bound developed in \cref{sec:main_result} to derive dimension reduction guarantees for regular sets.
Specifically, for a regular set $V \sset \R^N$, 
we find an upper bound on the sketching dimension $m$ ($m\ll N$), so that for a random matrix $\S:\R^N\rightarrow\R^m$ (either a sub-Gaussian matrix or an FJLT-like matrix, see \cref{def:gauss_rp} and \cref{def:sors}), with high probability it holds $\|\S\x\| \approx \|\x\|$ for every $\x\in V$.

\subsubsection{Backgrounds and Motivation}\label{sec:sketch-background}

As the dimension of datasets in modern applications grows, overcoming the curse of dimensionality becomes a key concern.
Among many attempts at dimension reduction (a.k.a. sketching),
one of the best known tools is the Johnson-Lindenstrauss transform \cite{johnson1984extensions}.
It compresses a finite set $\set{\x_i}_i$ down to a lower dimension by the map $\x_i\mapsto \S\x_i$, where $\S$ is a short-fat random matrix with properly rescaled i.i.d. Gaussian entires.
When the number of rows in $\S$, called the \it{sketching dimension}, is sufficiently large, $\S$ is an approximate isometry on this finite set, with high probability.

In more recent developments, researchers found better designs of the random matrix $\S$ that improve the efficiency of sketching. These include the sub-Gaussian sketch and the fast Johnson-Lindenstrauss transform (FJLT). The latter has been generalized to a broader class called the subsampled orthogonal matrix with random signs (SORS).
We give precise definitions below. For more details on theory and motivations of these sketching matrices, see e.g., \cite{achlioptas2003database-friendly,ailon2006approximate,martinsson2020randomized,woodruff2014sketching,vershynin2018high,iwen2022fast,dirksen2016dimensionality}.

\begin{definition}[sub-Gaussian sketch]
    \label{def:gauss_rp}
    Let $\G \in \R^{m\times M}$ be a random matrix such that the entries $G_{ij}$ are independent, mean zero, unit variance random variables with sub-Gaussian norm $\norm{G_{ij}}_{\psi_2} \leq \ga$.\footnote{This is the Orlicz norm with $\psi(x) = e^{x^2}-1$. If $\|X\|_{\psi_2} \leq \alpha$, its tail decays at least as fast as $e^{-\const(x/\alpha)^2}$, which justifies the name sub-Gaussian. See \cite{vershynin2018high}.}
    We call $\S = \frac{1}{\sqrt{m}}\G$ a \textup{sub-Gaussian random projection} with parameter $\ga$, denoted as $\S \in \subG(\ga)$. \footnote{We omitted the dimensions $m$ and $M$, since they will always be inferred from the context. Same for the next definition.}
\end{definition}

\begin{definition}[SORS sketch] \label{def:sors}
    Let $\m P \in \R^{m \times M}$ be the operator that uniformly randomly selects $m$ rows of a matrix (with replacement).
    Let $\m H \in \R^{M\times M}$ be a deterministic orthogonal matrix such that $\max_{ij} |H_{ij}| \leq \frac{\beta}{\sqrt{M}}$.
    Let $\m D \in \R^{M\times M}$ be a diagonal matrix with its diagonal populated by independent Rademacher random variables (i.e., $D_{ii}$ is 1 or -1 with equal probability).
    We call $\S = \sqrt{\frac{M}{m}}\m P\m H\m D$ a \textup{subsampled orthogonal matrix with random signs (SORS)} with parameter $\gb$,
    denoted $\S \in \sors(\gb)$.
\end{definition}

A more recent direction in randomized sketching concerns constructing $\S$ that is an approximate isometry on an \it{infinite} set $V$.
When $V$ is a linear subspace, the corresponding $\S$ has the name \it{subspace embedding}, and relevant theory is abundant (see e.g., \cite{woodruff2014sketching,martinsson2020randomized} for reviews).
When $V$ is a general set, developed theories (see e.g., \cite{vershynin2018high,oymak2018isometric,iwen2022fast}) state we can bound the sketching dimension for sketching operators defined in \cref{def:gauss_rp} and \cref{def:sors}, as long as we can control the Gaussian width of the set.

Here we use the covering number bound developed in \cref{sec:main_result} to bound the Gaussian width of bounded regular sets through the Dudley integral, which then derives sketching dimension bounds for regular sets using a sub-Gaussian sketch or an SORS sketch, so that
\begin{equation*}
    \|\S \x\| \approx \|\x\|~~\text{for all }\x\in V.
\end{equation*} 
In particular, we compute this bound for $V$ being polynomial images, as well as the image of $\l\circ p$ where $\l$ is a general Lipschitz map and $p$ is a polynomial.

Among others, one possible downstream task of dimension reduction is optimization:
\begin{equation}\label{eq:minf}
    \min_{\v x} \norm{f(\x)}, 
\end{equation}
where $f:\R^n \rightarrow \R^M$ ($M\gg n$) is the residual function we wish to bring close to zero by choosing $\x$.
One example is $f$ is an overdetermined polynomial system, with more equations than unknowns, where we seek an approximate solution $\x$.
Of particular interest is the following:
\begin{equation}\label{eq:min_fp}
    \min_{\v x \in U\sset \R^n} \norm{\ell\circ p(\x)}^2,
\end{equation}
where $p: \R^n \rightarrow \R^N$ is a polynomial map whose coordinate functions are of degree at most $d$, and $\ell: \R^N \rightarrow \R^M$ is a Lipschitz function.
When $\ell$ is the identity, this is the polynomial optimization problem, including low rank matrix and tensor decomposition, completion, etc.
Adding the $\ell$ function allows us to deal with non-$\ell_2$ loss functions.
For example, if we want a robust optimization using the Huber loss ($\ell_2$ when near the origin and $\ell_1$ when far away) \cite{huber1992robust}, we can take $\ell$ to be an entrywise function taking the square-root of Huber loss in each coordinate.
We will provide guarantees of sketching dimension so that $\|\S \l\circ p(\x)\|\approx \|\l\circ p(\x)\|$ for all $\x \in U$.
This allows us to optimize the reduced problem instead of the original problem.
For more on how randomized dimensionality reduction can accelerate optimization algorithms, please see Remarks~\ref{def:sketch_f} and \ref{rem:GN} below, as well as e.g., \cite{cartis2020scalable,cartis2022randomised,cartis2023scalable,berahas2020investigation,gower2019rsn,pilanci2017newton,ergen2019random}.

\subsubsection{New Results on Dimension Reduction using Covering Number}
\label{sec:ours_opt}

Consider the problem \eqref{eq:min_fp}, and denote $f = \ell\circ p$ for convenience.
In order to obtain a  guarantee on the dimension reduction,
we seek sketch operators $\m S$ such that with high probability over $\m S$,
\begin{equation} \label{eq:sf_approx_f}
    \norm{\m S f(\x)} \approx \norm{f(\x)} \quad \text{for all inputs } \x.
\end{equation}
The paper \cite{li2017near} provided such a result when $\ell = \Id$, $\S$ is the count sketch matrix, and under additional assumptions on the polynomial system $p$.
Compared to \cite{li2017near}, our framework applies whenever $\im(\pi\circ f)$ is regular.
We also allow other common random ensembles (sub-Gaussian and SORS) and require no further assumptions on $p$.
By switching to sub-Gaussian and SORS ensembles, our sketching dimension bound also improves \cite{li2017near} by significantly reducing the poly-logarithmic factors.

Next we specify exactly what is meant by $\norm{\S f(\x)} \approx \norm{f(\x)}$ in \eqref{eq:sf_approx_f}.
For convenience, write $x = (1 \pm \eps) y$ if $(1-\eps) y \leq x\leq (1+\eps) y$.
Roughly speaking, the next definition requires that when one of $\|\S f(\x)\|$ and $\|f(\x)\|$ is very small, the other must also be very small; and when they are both large, they are close to each other.

\begin{definition}[$(\gee, \gd, \gt)$ sketch] \label{def:sketch_f}
    Let $f$ be a function defined on $U$.
    Given a tolerance level $\gt \in (0, 1)$ and some $\gee, \gd \in (0, 1)$, 
    we say that a random matrix $\m S$ is an $(\gee, \gd, \gt)$ sketch of $f$ on $U$, if with probability at least $1 - \gd$ over $\S$, the followings hold for all $\x \in U$: (i) $\norm{\S f(\x)} \leq \gt$ implies $\norm{f(\v x)} \leq (1 + \eps) \gt$; (ii) $\norm{f(\x)} \leq \gt$ implies $\norm{\S f(\v x)} \leq (1 + \eps) \gt$; and (iii) $\norm{\S f(\x)} > \gt$ implies $\norm{\S f(\x)} = (1 \pm \eps) \norm{f(\x)}$.
\end{definition}

\begin{remark}
    If $\S$ is an $(\eps, \delta, \tau)$ sketch of $f$, 
    then it is useful for the downstream task of optimization \eqref{eq:minf}, because minimizing $\|\S f(\x)\|$ is almost equivalent to minimizing $\|f(\x)\|$.
    Indeed, let $\x^*$ be the obtained minimizer of $\|\S f(\x)\|$. Then either $\|\S f(\x^*)\| \leq \tau$, so that $\|f(\x^*)\|$ is also satisfactorily small, or that $\x^*$ is almost an optimal solution of the unsketched problem since 
    \begin{equation*}
        \|f(\x)\| \geq (1-\eps)^{-1} \|\S f(\x)\| \geq (1-\eps)^{-1} \|\S f(\x^*)\| \geq (1+\eps)(1-\eps)^{-1}\|f(\x^*)\|
    \end{equation*} 
    for all $\x \in U$.
\end{remark}

\begin{remark} \label{rem:GN}
    Consider the reduced problem of minimizing $\|\S f(\x)\|$ using Gauss-Newton method.
    In each iteration, we solve
    \begin{equation} \label{eq:gn-reduced}
        \Delta\x_S^* = \argmin_{\gD\x}\norm{\m S \m J \cdot \gD\x + \m S f(\x)},
    \end{equation}
    where $\m J$ is the Jacobian of $f$ at $\x$,
    and update $\x \gets \x + \Delta \x^*_S$.
    Many sketched optimization algorithms generate $\S$ anew for each solve of \eqref{eq:gn-reduced}, so that local to current iterate $\x$, with high probability, $\Delta\x_S^*$ is a good candidate for the unsketched, exact Gauss-Newton step vector that solves
    \begin{equation} \label{eq:gn}
        \Delta \x^* = \argmin_{\gD\x}\norm{\m J \cdot \gD\x + f(\x)}.
    \end{equation}

    By contrast, our theory would enable the drawing of a sketch $\S$ that satisfies \eqref{eq:sf_approx_f} and remains unchanged in every iteration of the optimization.  What are the potential benefits of this?  
    Using the local strategy, in each iteration there is a chance of failure that $\Delta\x_S^*$ is not a faithful approximation to $\Delta\x^*$.
    In addition, there are errors committed to the algorithm by using $\Delta \x_S^*$ in each iteration.
    Both the failure chance and the error can accumulate during the optimization run.  
    We refer the readers to \cite{pilanci2017newton} for an investigation on this matter.
\end{remark}

Since $\|\S f(\x)\| = (1\pm\eps)\|f(\x)\|$ is equivalent to $\|\S \pi\circ f(\x)\| = (1\pm \eps)$,
our task is then to sketch the set $W = \im(\pi\circ f)$.
For sketching an arbitrary set $W$, \cite{oymak2018isometric,iwen2022fast} and \cite{dirksen2016dimensionality} respectively proved a bound on sketching dimension $m$ in order for $\|\S\x\| = (1\pm\eps)\|\x\|$ for all $\x\in W$ with high probability, when $\S$ is respectively $\sors(\beta)$ and $\subG(\alpha)$ (see also \cite{vershynin2018high}). 
The bounds are expressed in the Gaussian width of $W$, so our goal is then to control the Gaussian width of this set.
Recall the Gaussian width of a set $W \sset \R^M$ is defined as
\begin{equation}
    \label{eq:def_gauss_width}
    \go(W) = \E_{\v g\sim \cN(0, \m I_M)}\sup_{\v x\in W}\ip{\v g}{\v x}.
\end{equation}
It is related to covering numbers through the \textit{Dudley's integral:}  
\begin{equation} \label{eq:dudley_gauss}
    \go(W) \leq 2 \int_0^{\diam(W)} \sqrt{\log\cN(W, \gee)}\, d\gee.
\end{equation}
Next, we give our result for generalized polynomial optimization \eqref{eq:min_fp}. 
A cleaner bound when $\ell = \Id$ comes afterwards.

\begin{remark} \label{rmk:genwidth}
    In fact, when $\v g$ is a vector of independent mean zero sub-Gaussian random variables, the corresponding ``width'' defined as in \eqref{eq:def_gauss_width} can also be bounded by the integral \eqref{eq:dudley_gauss} with a potentially constant in front.
    In particular, we will use this generalization in \cref{sec:gen_err},
    where $\v g$ is populated with Rademacher random variables, the corresponding width is called the Rademacher complexity.
    For more background, we refer the readers to \cite{vershynin2018high,mohri2018foundations}.  
\end{remark}

\begin{theorem} \label{thm:sketch_opt}
    Let $U \sset \R^n$ and $p:U\rightarrow\R^N$ be a polynomial map with coordinate functions $p_i$ of degree at most $d$ and $\sup_{\x \in U} |p_i(\x)| \leq t$ for each $i$.
    Let $\ell:\R^N\rightarrow \R^M$ be Lipschitz.
    Fix $\gee,\gt,\gd \in (0, 1)$.
    Denote $\gl = \max(\const, d N t \norm{\ell}_{\lip} \gt^{-1})$.
    Then $\m S \in \R^{m \times M}$ is a $(\gee, \gd, \gt)$ sketch of $\ell\circ p$ on $U$ if
    \begin{enumi}
        \item $\S \in \subG(\ga)$ and
            \begin{equation} \label{eq:subG_f}
                m \geq \const \ga^2 \gee^{-2} \left[n \cdot \log \Big(\gl\ga + \gl\eps\sqrt{(M + \log(1/\gd))/n}\Big) + \log (1/\gd)\right]; \text{ or }
            \end{equation} 
        \item $\S \in \sors(\gb)$ and
            \begin{equation}\label{eq:sors_f}
                m \geq 
                \const \gD \cdot \log^2\gD\cdot \log(N/\gd)
            \end{equation}
            where $\gD = \beta^2\eps^{-2} n \cdot \log\big(\gl\sqrt{M}\big)\log(1/\gd)$.
    \end{enumi}
\end{theorem}

\begin{proof}
    For any $\norm{\ell}_{\lip} = \theta > 0$, we can stretch $\ell$ and $p$ appropriately so that $\norm{\ell}_{\lip} = 1$ and the coordinate functions of $p$ are bounded by $t\theta$. 
    Since the claimed bound only involves the product $t\norm{\ell}_{\lip}$, without loss of generality we prove the assertions assuming $\ell$ is 1 Lipschitz.
    
    Denote $f = \ell\circ p$ and $V = p(U)$.
    Denote $F = f(U)\inter r \cB^M$ for some $r > 0$ chosen later, and denote $G = f(U) \setminus F$.
    Note that $\pi\circ \ell$ is $r^{-1}$ Lipschitz on $\ell^{-1}(G)$, so $\diam(\pi(G)) \leq \diam((\pi \circ \ell)(\ell^{-1}(G) \cap V)) \leq r^{-1} \diam(V) \leq 2 r^{-1} t\sqrt{N}$. 
    Thus, using \cref{thm:general_cover_number} and the Dudley integral \eqref{eq:dudley_gauss}, the Gaussian width of $\pi(G)$ is bounded by
    \begin{align*}
        \go(\pi(G)) 
        &\leq 
        \const \int_0^{\diam(\pi(G))} \sqrt{\log \cN(\pi(G), \eps)}\,d\eps
        \leq
        \const \int_0^{\diam(\pi(G))} \sqrt{\log \cN(V, r\eps)}\,d\eps\\
        &\leq
        \const \sqrt{n} \cdot \left( 
            \diam(\pi(G)) \sqrt{\pi} + 
            \diam(\pi(G)) \log^{1/2}\left( 
                \frac{d N\cdot r^{-1}t}{\diam(\pi(G))} \right)
        \right)\\
        &\leq
        \const \sqrt{n} \log^{1/2}\left(\max\set{\const, d N r^{-1} t}\right).
    \end{align*}
    In the second line, we first applied \cref{thm:general_cover_number} to bound the covering number, which is legitimate since $r \eps \leq \diam(V)$. 
    Then we applied \cref{lem:bnd_sqrt_ln} to bound the integral.
    In the last line,
    we used $\diam(\pi(G)) \leq \min\set{2, 2r^{-1}t\sqrt{N}}$ and separated the $\log^{1/2}$ term into a sum of two $\log^{1/2}$ terms.
    A maximum is used in the last line to handle the cases where $d Nr^{-1}t < 1$.
    By \cite[Corollary 5.4]{dirksen2016dimensionality}\footnote{Results therein guarantee $\norm{\S \y}^2 = (1 \pm\gee)\norm{\y}^2$, which implies $\norm{\S \y} = (1\pm\gee)\norm{\y}$ for $\gee \in (0, 1)$.} and \cref{thm:general_cover_number}, we get that for $\S \in \subG(\ga)$ and any $\gd_{\ga,G} \in (0, 1)$, if
    \begin{equation} \label{eq:mag}
        m \geq \const \ga^2\eps^{-2}\Big(
            n \log\left(\max\set{dNr^{-1}t, \const}\right)
            +
            \log (1/\gd_{\ga,G})
        \Big)
        =: \const m_{\ga, G},
    \end{equation}
    then $\norm{\S \y} = (1 \pm \eps) \norm{\y}$ for all $\y \in G$ with probability at least $1 - \gd_{\ga,G}$.
    By \cite[Theorem 9]{iwen2022fast}, for $\S \in \sors(\gb)$, if
    \begin{equation} \label{eq:mbg}
        m \geq \const \gD_r (\log^2 \gD_r) \log(N/\gd_{\gb, G}) := m_{\gb, G},
    \end{equation}
    where $\gD_r = \gb^2\gee^{-2}n \log(\max\set{\const, dNr^{-1}t})\log\gd_{\gb,G}^{-1}$,
    then $\norm{\S \y} = (1 \pm \eps) \norm{\y}$ for all $\y \in G$ with probability at least $1 - \gd_{\gb,G}$ for all $\gd_{\gb,G} \in (0, 1)$.

    Next we handle the set $F$.
    Our goal is to increase the sketching dimension $m$ to control the operator norm of $\S$, so that with $r \ll \gt$, for all $\y \in F$, we have $\norm{\S \y} \leq \gt$ with high probability.

    First consider $\S \in \subG(\ga)$.
    Let $u > 0$ to be specified later, and we choose $r = \frac{1}{\ga u} \gt$.
    Then for all $\y \in F$, in order for $\norm{\S\y} > \gt$, we must have $\norm{\S} > \ga u$.
    Let $c_1, c_2$ be appropriate constants in \cref{lem:subG_norm}.
    By \cref{lem:subG_norm},
    for arbitrary $\gd_{\ga,F} \in (0, 1)$, 
    if we take $u$ such that
    \begin{equation}\label{eq:u_ineq}
        (c_1 u^2 - c_2)^{-1} (M + \log(1/\gd_{\ga, F})) \leq \ga^2\gee^{-2}n,
    \end{equation}
    then $\norm{\S} > \ga u$ happens with probability at most $\gd_{\ga, F}$ provided that $m \geq \ga^2\gee^{-2} n$.
    Note that \eqref{eq:u_ineq} is guaranteed if
    \begin{equation}\label{eq:u_bnd}
        u \geq c_3\left(1 + \gee\ga^{-1}\sqrt{(M + \log(1/\gd_{\ga, F})) / n}\right)
    \end{equation}
    for some $c_3 > 0$.
    The corresponding value of $r^{-1}$ is 
    \begin{equation} \label{eq:r_bnd}
        r^{-1} \geq c_3 \gt^{-1}\left(\ga + \gee\sqrt{(M + \log(1/\gd_{\ga, F})) / n}\right) \geq c_3\ga \gt^{-1}.
    \end{equation}
    Since entries of $\wh{\m S} = \sqrt{m}\m S$ has a unit variance, their sub-Gaussian norm $\ga$ is lower bounded by an absolute constant.
    In particular, by choosing $c_3$ sufficiently large, we can guarantee that $r < \gt$, and thus $F \sset (\gt\cdot \cB^N)$.
    Therefore, with $\gd_{\ga, F} = \gd_{\ga, G} = \gd/2$, the union bound shows that with $m = \const m_{\ga, G}$, in which $r$ is given by \eqref{eq:r_bnd}, we simultaneously have
    \begin{equation} \label{eq:sketch_condition}
        \begin{gathered}
            \norm{\S \y} = (1 \pm \eps/2) \norm{\y} ~\text{for}~ \y \in G \supseteq (\gt\cdot \cB^M)^c, ~\text{and}~\\
            \norm{\S\y} \leq \gt ~\text{for}~ \y \in F \sset (\gt\cdot \cB^M).
        \end{gathered}
    \end{equation}
    The conditions above thus shows that for all $\x \in U$, when $\norm{f(\x)} \leq \gt$,$\norm{\S f(\x)} \leq (1 + \eps) \gt$, and when $\norm{\S f(\x)} > \gt$, $\norm{\S f(\x)} = (1 \pm \eps) \norm{f(\x)}$.
    Finally, if $\norm{f(\x)} > (1 + \eps) \gt$, then $\norm{\S f(\x)} > (1-\eps/2)(1 + \eps) \gt > \gt$ since $\eps < 1$.
    The contrapositive completes the proof that $\S$ is a $(\gee, \gd, \gt)$ sketch of $f$.
    The dimension bound \eqref{eq:subG_f} is a simplification of $m_{\ga, G}$ using the fact that $\log(\gl \ga) = \gO(1)$ with proper choice of lower bounding constant $\const$ in the definition of $\gl$.

    Next we deal with $\S \in \sors(\gb)$.
    We choose $r = \gt/\sqrt{M}$.
    Note that $\norm{\S} \leq \sqrt{M}$. 
    This guarantees for all $\y \in F$, $\norm{\S \y} \leq \gt$.
    Plugging this $r$ into \eqref{eq:mbg} and setting $\gd_{\gb, G} = \gd$, we see that condition \eqref{eq:sketch_condition} hold with probability at least $1-\gd$ when $m$ satisfies the bound in \eqref{eq:sors_f}.
    Hence, $\S$ is a $(\gee, \gd, \gt)$ sketch of $f$ as claimed.
\end{proof}

The bounds in \cref{thm:sketch_opt} can be greatly simplified if in fact we have a polynomial optimization problem ($\ell = \Id$), as shown next.
    In particular, we no longer have any $M$ dependence in the sketching dimension nor any $t$ dependence from the coordinate functions of $p$.
    
\begin{corollary}\label{cor:polyopt}
    Let $p: \R^n \rightarrow \R^N$ be a polynomial map with coordinate functions of degree at most $d$.
    Denote $N' = \min\set{N, n^d}$.
    Then for any $\gee,\gd\in(0, 1)$, $\m S \in \R^{m \times M}$ is a $(\gee, \gd, 0)$ sketch of $p$ if
    \begin{enumi}
        \item $\S \in \subG(\ga)$ and
            \begin{equation}\label{eq:subG_p}
                m \geq \const \ga^2 \gee^{-2} \Big(n \cdot \log (ndN') + \log(1/\gd)\Big); \text{ or }
            \end{equation} 
        \item $\S \in \sors(\gb)$ and
            \begin{equation}\label{eq:sors_p}
                m \geq \const \gD \cdot \log^2\gD \cdot \log(N'/\gd)
            \end{equation}
            where $\gD = \beta^2\eps^{-2} n\cdot \log(ndN') \log(1/\gd)$.
    \end{enumi}
\end{corollary}

\begin{proof}
    Following the computation of Gaussian width in the previous proof, this time $V = \pi(\im(p))$ and it is contained in a subspace of dimension $\min\set{N, \binom{n + d}{d}} = \cO(N')$. Then
    \begin{equation*}
        \go(\pi(\im(p))) \leq 2 \int_0^2 \sqrt{\log\cN(V, \gee)}\,d\gee \leq \const \sqrt{n\log(ndN')},
    \end{equation*}
    where in the last step we used \cref{lem:polynom_regularity}.
    Thus, using \cite[Corollary 5.4]{dirksen2016dimensionality}, we get the desired result.
    The bound for $\S \in \sors(\gb)$ follows directly from \cite[Theorem 9]{iwen2022fast} using the bound on the Gaussian width above.
\end{proof}

\begin{remark}
    We emphasize that our sketching bounds based on \cref{thm:general_cover_number} are (essentially) \textup{oblivious} to $\ell\circ p$.
    They only require a bound on the Lipschitz constant of $\ell$, a degree bound on the coordinate functions of $p$, and a bound on the coordinate functions on $U$.
    In comparison, there are existing works \cite{federer1959curvature,iwen2022fast}  that in principle could be used to give global guarantees \eqref{eq:sf_approx_f} by a similar argument.  However, these require nontrivial geometric descriptors of the image $p(U)$, such as the reach and the volume, which often cannot be estimated efficiently.
\end{remark}

\subsection{Generalization Error of Neural Networks} \label{sec:gen_err}

Covering numbers are also relevant to learning theory. For a thorough and relatively reader-friendly introduction to the subject, see e.g., \cite{mohri2018foundations}.
The generalization error of a classification algorithm measures the gap between accuracy on the training set and on unseen test data.
Suppose the data feature $\v x$ and label $y$ follow a distribution $(\x, y) \sim \cD$.
For a classifier $f$ and loss function $\l$, the expected risk of $f$ is defined as
\begin{equation*}
    R(f; \l) = \E_{(\v x, y) \sim \cD}\l(f(\v x), y).
\end{equation*}
When provided an i.i.d. sample $S = \set{(\v x_i, y_i)}_{i = 1}^n$ from $\cD$, the empirical risk is
\begin{equation*}
    R_S(f;\l) = \frac{1}{n}\sum_{i = 1}^n \l(f(\v x_i), y_i). 
\end{equation*}
We wish to control the generalization error 
\begin{equation}
    \gerr_S(\cF;\ell) := \sup_{f\in \cF} \, R(f; \ell) - R_S(f; \ell), 
\end{equation}
so we can expect good performance on unseen test data.
For more details, see
\cite{anthony1999neural,vershynin2018high}.

There is a long history in finding optimal bounds for generalization errors.
One of the best studied setups is the binary classification case, where $f \in \cF$ has a range $\set{-1, 1}$, $y \in \set{-1, 1}$, and the loss function is $\ell_{\text{bin}}(f(\x), y) = \ind{f(\x) \neq y}$.
For such a setup, it is known that (see e.g., \cite[Theorem 8.4.4]{vershynin2018high})
\begin{equation*}
    \gerr_S(\cF;\ell) \leq \const \sqrt{\vc(\cF)/n}
\end{equation*}
with high probability, where $\vc(\cF)$ is the VC dimension of $\cF$ and $n$ is the sample size.
Hence, many bounds for VC dimensions of many hypothesis classes have been proposed, such as  deep neural networks with piecewise polynomial activations \cite[Theorem 8.8]{anthony1999neural} and \cite{bartlett2019nearly}.

Instead we consider another popular setup, common in multi-label classification, where the output of $f$ and the loss function $\ell$ are continuous.
A key tool to control the generalization error in the continuous setup is the \it{Rademacher complexity}: 
\begin{equation}
    \label{eq:rad_complexity}
    \fR_S(\cF;\l) = \E_{\v \gs} \sup_{f \in \cF} \frac{1}{n} \sum_{i = 1}^n \gs_i \cdot \l(f(\v x_i), y_i),
\end{equation}
where entries $\gs_i$ of $\v \gs$ are independent Rademacher random variables ($\sigma_i = \pm1$ with equal probability).
By \cite[Theorem 3.3]{mohri2018foundations}, if $\l$ has a range in $[0, 1]$, then with probability at least $1-\gd$ over the sample $S$, the generalization error of $\cF$ can be bounded by
\begin{equation} \label{eq:gen_err}
    \gerr_S(\cF;\ell) \leq 2\fR_S(\cF;\l) + 3\sqrt{\frac{\log(2/\gd)}{2n}}.
\end{equation}
The Rademacher complexity \eqref{eq:rad_complexity} is analogous to the Gaussian width \eqref{eq:def_gauss_width}.
Indeed, it can also be bounded through the Dudley integral (see \cref{rmk:genwidth}):
\begin{equation}\label{eq:dudley_rad_complexity}
    \fR_S(\cF;\l) \leq \const n^{-1}\int_0^{\sqrt{n}}\sqrt{\log \cN(\l \circ \cF(\m X), \gee)}\, d\gee, 
\end{equation}
where $\m X = (\x_1|\ldots|\x_n) \sset \R^{d \times n}$, $f(\m X) = (f(\x_1)|\ldots|f(\x_n)) \in \R^{D \times n}$ and
$$\mathcal{F}(\X) = \set{f(\m X): f\in \cF} \subseteq \R^{D \times n}.$$
In \eqref{eq:dudley_rad_complexity}, we are also assuming that $\ell$ takes values in $[0,1]$.

In the sequel, we derive generalization bounds for deep neural networks by controlling this covering number by \cref{thm:general_cover_number}.
We cover fully connected or convolutional layers, with rational or ReLU activations.
For rational activations analyzed in \cref{sec:gen_err_rat}, we show the image $\cF(\X)$ is a regular set.
Then for ReLU activations analyzed in \cref{sec:gen_err_common}, we approximate $\cF(\X)$ by the image of a rational network and apply the results in \cref{sec:gen_err_rat}.

\subsubsection{Generalization Error of Rational Neural Networks}\label{sec:gen_err_rat}

Here we bound the generalization error of rational neural networks \cite{boulle2020rational}.  
These are networks having rational functions (i.e., ratio of polynomials) as their activations.
They can approximate ReLU networks \cite{molina2019pad, boulle2020rational, telgarsky2017neural}.  Beyond that,  rational activations are naturally trainable since they are parameterized by coefficients.  As such, it has been observed that rational networks can be performant and accelerate training process in data-driven applications compared to other activations \cite{boulle2020rational,huang2023fast}. 
We recall the definition of a rational function.

\begin{definition}
    \label{def:rational}
    We say that a function 
    $\gr$ from an open dense subset of $\R^d$ to $\mathbb{R}$ is a \textup{rational function} if $\gr(\x) = p(\x)/q(\x)$ for some polynomials $p$ and $q$.
    If $\deg(p) \leq \ga$ and $\deg(q) \leq \gb$, we denote $\gr \in R^\ga_\gb$.
    We say that a map $f$ from an open dense subset of $\R^d$ to $\R^D$ is in $R^\ga_\gb$ if all coordinate functions of $f$ are in $R^\ga_\gb$.
\end{definition}

The next definitions formalize the two classes of networks to be analyzed. 

\begin{definition}
    [rational fully connected neural network]\label{def:ratnn}
    For $i = 1,\ldots,L$,
    let $\gr_i \in R^{\ga_i}_{\gb_i}$ be rational activation functions; let $\m A_i \in \R^{d_i \times d_{i-1}}$ be weight matrices; and let $\v b_i \in \R^{d_i}$ be bias vectors.
    The \textup{rational fully connected neural network} is the composed map from $\R^{d_0}$ to $\R^{d_L}$:
    \begin{equation}
        \label{eq:def_ratnn}
        f(\x) = \gr_{L}\Big[\A_{L} \big(\ldots \gr_1(\A_1 \x + \b_1)\ldots\big) + \b_L\Big].
    \end{equation}
    Denote $\v d = (d_0,\ldots,d_L)$, $\v \ga = (\ga_1,\ldots,\ga_L)$, and $\v \gb = (\gb_1,\ldots,\gb_L)$.
    We denote this hypothesis class as $\ratnn(L, \v \ga, \v\gb, \v d)$.
    Denote the function of the first $i$ layers as $f_i = \gr_i\circ \A_i\circ\ldots$.
\end{definition}

\begin{definition}
    [rational convolutional neural network]\label{def:ratcnn}
    For $i = 1,\ldots,L$,
    let $\gr_i \in R^{\ga_i}_{\gb_i}$ be rational activation functions; let $\t K_i \in \R^{c_i \times c_{i-1} \times k\times k}$ be convolution kernels\footnote{Here $d_i$ represents the total size of the ``image" with $c_i$ channels, and $\t K_i$ possibly involves strides and zero or constant paddings.} that map input in $\R^{d_{i-1}}$ to $\R^{d_i}$, denoted by $\v x \mapsto \t K_i \x$; and 
    let $\v b_i \in \R^{d_i}$ be bias vectors.
    The \textup{rational convolutional neural network} is the composed map from $\R^{d_0}$ to $\R^{d_L}$:
    \begin{equation}
        \label{eq:def_ratcnn}
        f(\x) = \gr_{L}\Big[\t K_{L} \big(\ldots \gr_1(\t K_1 \x + \b_1)\ldots\big) + \b_L\Big].
    \end{equation}
    Denote $\v d = (d_0,\ldots,d_L)$, $\v c = (c_0,\ldots,c_L)$, $\v \ga = (\ga_1,\ldots,\ga_L)$, and $\v \gb = (\gb_1,\ldots,\gb_L)$.
    We denote this hypothesis class as $\ratcnn(L, \v \ga, \v\gb, \v c, \v d)$.
    Denote the function of the first $i$ layers as $f_i = \gr_i\circ \t K_i\circ\ldots$.
\end{definition}

Our next theorem upper bounds the covering number of $\cF(\m X)$ for four hypothesis spaces within $\ratnn(L, s\ind{}, s\ind{}, \v d)$ and $\ratcnn(L, s\ind{}, s\ind{}, \v c, \v d)$.  
We denote these hypothesis spaces as $\cF_{\ratnn}$ and $\cF_{\ratcnn}$ with \textit{fixed} activations, as well as the variants $\cF^T_{\ratnn}$ and $\cF^T_{\ratcnn}$ with \textit{trainable} activations. 

\begin{theorem} \label{thm:covering_ratnn}
    Let the hypothesis classes be defined as in the above paragraph.
    Denote $N$ the number of parameters in the network and 
    suppose $\im(\gr_L) \in [-t, t]$.
    Then for $\cF = \cF_{\ratnn}$ or $\cF_{\ratnn}^T$ and any $\gee \in (0, \diam(\cF(\m X))]$,
    \begin{equation}\label{eq:cover_ratnn}
        \sup_{\m X \in \R^{d_0 \times n}} \cN(\cF(\m X), \gee) \leq 
        N\log(t/\eps) + (N + 1)\log\left(\const N (nd_L)^{5/2} s^L \prod\nolimits_{j = 1}^{L-1}d_j \right),
    \end{equation}
    where $\const$ can possibly be different for $\cF_{\ratnn}$ and $\cF_{\ratnn}^T$.
    For $\cF = \cF_{\ratcnn}$ or $\cF_{\ratcnn}^T$ and any $\gee \in (0, \diam(\cF(\m X))]$,
    \begin{equation}\label{eq:cover_ratcnn}
        \sup_{\m X \in \R^{d_0 \times n}} \cN(\cF(\m X), \gee) \leq 
        N \log(t/\gee) + (N + 1)\log\left(\const N (nd_L)^{5/2} s^L k^{2(L-1)} \prod\nolimits_{j = 1}^{L-1}c_{j} \right),
    \end{equation}
    where $\const$ can possibly be different for $\cF_{\ratcnn}$ and $\cF_{\ratcnn}^T$.
\end{theorem}

 \cref{thm:covering_ratnn} relies on \cref{thm:general_cover_number}. 
Indeed, below we show that $\cF(\m X)$ is a $(K, N)$ regular set for bounded $K$.
Let $\gL$ be a tuple of all involved parameters. 
It consists of the entries of $\m A_i, \t K_i, \v b_i$, and in the case the activations are trainable, also the coefficients in each $\gr_i$.
Obviously, each choice of $\gL$ defines a network $f^\gL \in \cF$.
Define the evaluation map $\m X^*(\gL) := f^\gL(\m X) \in \R^{d_L\times n}$ and similarly $\m X^*_{i}(\gL) := f_i^{\gL}(\m X)$ for the first $i$ layers.
Note that $\m X^*$ is a rational map in $\gL$. 
By \cref{lem:polynom_regularity}, we just need a degree bound on the coordinate functions of $\m X^*$ to control $K$.

\begin{proof}[proof of \cref{thm:covering_ratnn}]
    First note that $\m X^*$ is well defined on all $\gL$ such that $\im(\gr_L) \in [-t, t]$.
    Note that $\cF(\m X) = \im(\m X^*)$ is a regular set.
    Indeed, let $\m X^*(\gL)_j = \frac{p_j(\gL)}{q_j(\gL)}$, $j = 1,\ldots,nd_L$.
    By \cref{lem:degree_ratnn}, each $p_j$ and $q_j$ has a degree at most $d_f$, which is computed therein.
    By \cref{lem:polynom_regularity}, $\cF(\m X)$ is $((2nd_L d_f + 1)^{N+1}, N)$ regular. 
    Using \cref{thm:general_cover_number}, we have 
    \begin{equation*}
        \cN(\cF(\m X), \gee)
        \leq
        N\log\left(\frac{\const Nt (nd_L)^{3/2}} {\gee}\right)
        +
        (N+1) \log(\const nd_L d_f).
    \end{equation*}
    Plugging in the bound on $d_f$ in \cref{lem:degree_ratnn} for corresponding $\cF$ gives the claimed result.
\end{proof}

Combining \cref{lem:polynom_regularity} and \eqref{eq:gen_err}, we derive the following generalization error bounds on the four hypothesis spaces with a Lipschitz loss function.

\begin{corollary}
    \label{cor:gen_err_rational}
    Let the hypothesis spaces be the same as in \cref{thm:covering_ratnn}. 
    Suppose the loss function $\l$ is Lipschitz with constant $\norm{\l}_{\lip}$\footnote{We mean $\l$ is Lipschitz with constant $\norm{\ell}_{\lip}$ in $\x$ for each $y$.} and range $[0, H]$, and $\im(\gr_L) \in[-t, t]$. 
    Denote $N$ the number of parameters and 
    $\ga = \min\set{H, t\norm{\ell}_{\lip}\sqrt{d_L}}$.
    Then 
    \begin{equation}\label{eq:nice-one}
        \sup_{\m X \in \R^{\din\times p}}\fR_S(\cF;\l) \leq \const \ga\sqrt{\frac{N}{n}}\left(\sqrt{\pi} + \sqrt{\log\frac{t W\norm{\l}_{\lip}}{\ga\sqrt{n}}}\right),
    \end{equation}
    where $W$ is the number in the second $\log(\cdot)$ term in \eqref{eq:cover_ratnn} and \eqref{eq:cover_ratcnn} for the corresponding class $\cF$.
    
    In particular, write $D = \max_{i\leq L-1} d_i$ for ${\ratnn}$ and its trainable variant, and $D = k^2\max_{i\leq L-1} c_i$ for ${\ratcnn}$ and its trainable variant.
    Suppose ($\ast$): $t$, $H$, and $\norm{\l}_{\lip}$ are $\cO(1)$ and $\log d_L$, $\log n$ are $\cO(L \log(sD))$. 
    Then
    \begin{equation} \label{eq:bnd_ratnn_with_assumption}
        \sup_{\m X \in \R^{d_0\times n}}\fR_S(\cF;\l) = \cO\left(n^{-1/2} \cdot \sqrt{N L \log(sD)}\right).
    \end{equation}
    
    If in fact the classification loss is $\l(\v x, \v y) = \crsent(\softmax(f(\v x)), \v y)$, where $y_i = 1$ for data in class $i$ and $y_i = 0$ otherwise, then under ($\ast$):
    \begin{equation}
        \sup_{\m X \in \R^{d_0\times n}}\fR_S(\cF;\l) = \cO\left(n^{-1/2} \cdot (t + \log d_L) \sqrt{N L \log(sD)}\right).
    \end{equation}
    Equation \eqref{eq:gen_err} then gives a bound for the generalization error for all cases above.
\end{corollary}

\begin{remark} \label{rmk:n_assumption}
    Assumption $(\ast)$ in \cref{cor:gen_err_rational} is a mild  assumption
    in applications involving nontrivial networks. 
    For example, suppose $L = 10$, $D = 10$, $s = 3$. Then order for $\log n > 2 L \log(sD)$ to hold one would require $n \geq 30^{20} \approx 3.5\times 10^{29}$ many data samples. 
    This is far beyond realistic sizes of  training sets.
\end{remark}

\begin{proof}[proof of \cref{cor:gen_err_rational}]
    Since $\l$ is Lipschitz on $[0, H]$, $\l/H$ is Lipschitz on $[0, 1]$ with norm $\norm{\l}_{\lip}/H$. 
    Using the Dudley integral \eqref{eq:dudley_rad_complexity},
    \begin{equation*}
        \fR_S(\cF;\ell) 
        \leq 
        \const \frac{H}{n}\int_0^{\sqrt{n}}\sqrt{\log \cN\left(\frac{\l(\cF(\m X), \v y)}{H},\,\gee\right)} \,d\gee
        \leq
        \const \frac{H}{n}\int_0^{\frac{\ga}{H}\sqrt{n}}\sqrt{\log \cN\left(\cF(\m X),\,\frac{H}{\norm{\l}_{\lip}}\gee\right)} \,d\gee.
    \end{equation*}
    Bounding the covering number in the last term by \cref{thm:covering_ratnn} with\footnote{The relaxed bound using the last term is sufficient for obtaining the desired big-$\cO$ scaling.}
    \begin{equation*}
        \log \cN\left(\cF(\m X),\gee\right) \leq \const N \log(t W /\gee),
    \end{equation*}
    and further bounding the integral using \cref{lem:bnd_sqrt_ln}, we get
    \begin{equation*}
        \fR_S(\cF; \l) \leq 
        \const \ga \sqrt{\frac{N}{n}}\left(\sqrt{\pi} + \sqrt{\log\frac{t W\norm{\l}_{\lip}}{\ga\sqrt{n}}}\right).
    \end{equation*}
    This completes the proof for the first part.
    In the other two parts, we have $\ga = \cO(1)$.
    The second part then follows directly by plugging in the value of $W$ in first part.
    For the third part, recall that for all $\v y$, $|\pd{\l}{x_i}| \leq 1$ and since $f(\v x)\in [-t, t]^{d_L}$, $\l(\x, \y) \in [0,~ \log (d_L e^t/e^{-t})] = [0,~ 2t + \log d_L]$.
    So $\norm{\l}_{\lip} \leq \sqrt{d_L}$ and $H \leq 2M + \log d_L$. Plugging this into the result from the first part we obtain the claim.
\end{proof}

\begin{remark} 
    How good are the generalization bounds in \cref{cor:gen_err_rational}?
    To our knowledge, there exists no prior result in our exact setting.
    However, compared to the bound in \cite[Theorem 8.4]{anthony1999neural} for hypotheses $f$ that require only $b$ arithmetic and comparison operations, setting $t = \cO(1)$, $s = \cO(1)$, and $d_i = D$ for $i\in[L]$, we improved that bound from $\cO(n^{-1/2}\cdot LD^2)$ to $\cO(n^{-1/2} \cdot LD\sqrt{\log D})$.
    The improvement is even more remarkable for convolutional networks, where $b$ the number of flops needed to evaluate $f$ is far greater than $N$.
    Although these bounds are derived under a different setup and are not directly comparable, we take this comparison as a certificate of the nontriviality of \cref{cor:gen_err_rational}.
\end{remark}

Note that this section did not require boundedness of the data $\m X$, nor boundedness of weight matrices $\m A$ or kernels $\t K$, nor Lipschitz properties of the activations. 
All we used is a bound of the final output, $\im(\gr_L)$. 
 That is necessary anyways for $\cF(\m X)$ to have a finite covering number.

\subsubsection{Generalization Error of ReLU Networks}\label{sec:gen_err_common}

Now we shift attention to ReLU-activated neural networks.
We only treat fully connected networks here, since the purpose is to demonstrate another quick application of \cref{thm:general_cover_number}.

The standard setup for generalization error analysis of  ReLU network is the following.
The network is fixed of shape $(L, \v d)$ where $\v d = (d_0,\ldots, d_L)$, with $d_0$ being the input dimension and $d_L$ being the output dimension.
The hypothesis class consists of networks with bounded weights and biases.
Here we take $\norm{(\m A_i | \v b_i)}_{\infty} \leq \go_i$.
The input data $\x$ is also assumed to be bounded.
We take $\norm{\x}_{\infty} \leq 1$.
This ensures the output of the network is a bounded set.
Finally, the loss function $\l$ is Lipschitz.
Denote $\v \go = (\go_1,\ldots,\go_L)$, and we denote the hypothesis class as $\cF = \relu(L, \v d, \v \go)$.
Our bound on the Rademacher complexity, and hence the generalization error, is the following.

\begin{theorem}[restate = generrReLU, name = generalization error or ReLU network]\label{thm:gen_err_common_full}
    Let $\cF = \relu(L, \v d, \v \go)$ with $d_i \geq 2$, and $\go_i \geq 2$.
    Suppose $\l$ is a Lipschitz loss function having range $[0, H]$. 
    Denote $N$ the number of parameters in the network, $\ga = \min\set{H, \,\norm{\ell}_{\lip}\sqrt{d_L} \prod_{i = 1}^L \go_i}$, and $\gb = \ga^{-1}\norm{\ell}_{\lip}\sqrt{d_L} \prod_{i = 1}^L \go_i$.
    Then
    \begin{multline*}
        \sup_{\m X \in \R^{d_0\times n},~\norm{\m X}_{\max} \leq 1}\fR_S(\cF;\l)
        \leq \\
        \const \ga \sqrt{N/n} \cdot
        \left(
            \log\gb 
            +
            \log n
            + 
            \sum_{i = 1}^L \log d_i
            +
            L \log_+\log_+\left(\gb\sqrt{n/d_L}\right)
        \right)^{1/2},
    \end{multline*}
    where $\log_+(x) = \max\set{0, \log(x)}$.
    In particular, if $\v d = d\ind{}$, $\v \go = \go\ind{}$, and suppose the loss function is such that $\norm{\l}_{\lip}$ and $H$ are $\cO(1)$ constants, and the data size satisfies $\log n = \cO(L \log (dL))$ (see \cref{rmk:n_assumption}), then the bound simplifies to
    \begin{equation} \label{eq:gerr_relu_lite}
        \sup_{\m X \in \R^{d\times n},~\norm{\m X}_{\max} \leq 1}\fR_S(\cF;\l)
        = \\
        \cO\left(
            n^{-1/2} \cdot dL \sqrt{\log (d L \go)} 
        \right).
    \end{equation}
\end{theorem}

\begin{proof}[Proof sketch (see \cref{app:gen_err_common} for full details)]
    We use rational functions to approximate the ReLU function at each layer (\cref{lem:rat_approx_relu}).
    This is possible since the input to each layer is bounded.
    By controlling approximation errors at each layer, we obtain a rational network that approximates the original ReLU network (\cref{lem:propagation} and \cref{lem:degree_rat_approx}).
    We cover the image of the ReLU network by covering the image of this approximating rational map.
    Using our result in the rational network case, namely \eqref{eq:nice-one} in \cref{cor:gen_err_rational}, leads to a bound on Rademacher complexity of the ReLU network.
\end{proof}

Let us compare the simplified bound \eqref{eq:gerr_relu_lite} to existing literature.
The result in \cite{bartlett2019nearly} for VC dimension of networks with ReLU activations implies a generalization error bound of order $n^{-1/2} \cdot \sqrt{N L \log N}$.
Assuming $\go = \cO(1)$ for the binary classification problem in order for the output to be within $[-1, 1]$, our bound \eqref{eq:gerr_relu_lite} is also of order $n^{-1/2} \cdot \sqrt{N L \log N}$, matching the result in \cite{bartlett2019nearly} for the binary case.
Despite differences in the setups (binary vs. continuous), we interpret this comparison as confirming the quality of bounds in \cref{thm:gen_err_common_full}.

Other works have studied the same setup as ours.
See \cite{lin2019generalization} for an overview of some prior results.
Bounds obtained there were later improved in \cite{long2019generalization} by a different approach.  
To our knowledge, \cite{long2019generalization} represents the current state-of-the-art for ReLU networks.
Suppose $\log L = \cO(\log d)$ and $\go = \gO(d)$, which is the case if entries of $\m A$ in the hypothesis space are allowed to have a magnitude of 1.
Then the bound \eqref{eq:gerr_relu_lite} matches the one in \cite{long2019generalization}.

\begin{remark}
Intuitively, one might not expect our approach based on rational function approximations to produce a result comparable to the state-of-the-art for ReLU networks.
When we use a rational function in $R^s_s$ to approximate the ReLU function,
then the bound we obtain is controlling the worst case generalization error for $\cF_{\ratnn}$ over arbitrary fixed activations in $R_s^s$.
Such networks are expected to be more complicated than ReLU-activated networks of the same structure.
Nonetheless, our main result \cref{thm:general_cover_number} does lead to a good result here. 
\end{remark}

\section{Conclusion}

In this paper, we derived a covering number bound $\cN(V, \gee)$ for bounded regular sets $V$ (\cref{def:reg_set}), which include images of polynomial maps, real algebraic varieties, and more generally semialgebraic sets in Euclidean space.
The bound (\cref{thm:general_cover_number}) improves the best known prior general result in \cite{yomdin2004tame}.
As an immediate corollary, we bounded the volume of tubular neighborhoods of polynomial images (\cref{cor:tubvol_image}), as well as varieties and semialgebraic sets.
Our main result on the covering number (\cref{thm:general_cover_number}) is asymptotically tight as $\gee \rightarrow 0$, and we expect the dependencies on other quantities in the bound are also in their optimal form in general.
Compared to existing covering number bounds developed for smooth manifolds, ours for regular sets only requires an inspection of the number of variables and the degree of the involved polynomials, and does not require nontrivial information such as the volume or reach of $V$. 
Indeed, the paper's proof by slicing leverages {global} properties of algebraic varieties, rather than {local} properties of manifolds.

We applied \cref{thm:general_cover_number} to a number of computational domains.  
Firstly we bounded the covering number on tensors of low CP rank (\cref{thm:covering_low_cp_rank}).
That result is near optimal in the low CP rank regime ($r \leq \min_i n_i$), often of most interest in practical tensor methods.
Moreover, we applied \cref{thm:general_cover_number} to sketching images of polynomials (and other regular sets), applied it to nonlinear polynomial optimization problems, and bounding the generalization error of neural networks.
In polynomial optimization, for $\m S$ a sub-Gaussian random projection or subsampled orthogonal matrix with random signs (\cref{thm:sketch_opt}),
we derived a sketching dimension bound ensuring that $\S$ is an $(\gee, \gd, \gt)$ sketch operator (\cref{def:sketch_f}) of $\ell\circ p$ for a Lipschitz residual function $\ell$.
This bound is greatly simplified when $\ell$ is identity (\cref{cor:polyopt}).
For rational neural networks of depth $L$ and width $d$ (fully connected or convolutional) with possibly trainable activations (which can greatly reduce training time), we derived a bound on its generalization error that is independent of the norms of weight matrices and that of the input vector (\cref{thm:covering_ratnn}).
The bound is $\cO(dL/\sqrt{n})$ up to logarithmic factors.
We also tested our theory by deriving a generalization error bound on ReLU networks (\cref{thm:gen_err_common_full}) using the approximation properties of rational functions. 
Under reasonable assumptions, that result matches the best prior result to our knowledge \cite{long2019generalization}, and shows dependence on the choice of architecture of the network.
In general, we expect that several more useful applications of \cref{thm:general_cover_number} in applied and computational mathematics should be possible.

\hfill\newline 
\noindent
\textbf{Acknowledgements.} 
The authors are grateful to Saugata Basu and Nicolas Boumal for helpful pointers to the literature.
Y.Z. was supported in part by a Graduate Continuing Fellowship and a National Initiative for Modeling and Simulation (NIMS) graduate student fellowship at UT Austin. J.K. was supported in part by
NSF DMS 2309782, NSF DMS 2436499, NSF CISE-IIS 2312746, DE SC0025312, and start-up grants from the College of Natural Science and Oden Institute at UT Austin.

\printbibliography

\hfill\\
\appendix

\section{Proofs of Properties of Regular sets}\label{app:mainthm}

We begin by introducing necessary notations.
Let $A\in \overline{G}_N^k$ be parametrized by $A = (\x, P)$, where $\x \in \R^{N-k}$ is the affine vector and $P \cong \R^k$ the tangent space of $A$.
For each $\x$, denote $\nu_N^k = \mu_N^k(\x,\cdot)$ the $O(N)$ invariant measure on $G_N^k$ (the Grassmannian of $k$ dimensional \textit{linear} subspaces of $\R^N$), and for each $P$, denote $\gl_{N-k} = \mu_N^k(\cdot, P)$ the Lebesgue measure on $P^\perp$.

The following lemma shows an expected result. Let $k\leq m\leq n$, and $E \sset G_n^k$ be a $\nu_n^k$-null set.
Then for almost all $f \in G_n^m$ of higher dimension, almost all $k$ dimensional subspaces of $f$ are not an element of $E$.

\begin{lemma} \label{lem:gr_integral}
    Let $E\sset G_n^k$ be such that $\nu_n^k(E) = 0$.
    Then for each $m \geq k$, there is a $\nu_n^m$ null set $F$ in $G_n^m$, so that for all $f \notin F$, the set $E_f = \set{e \in G_n^k: e \sset f, e \in E}$ satisfies $\nu_m^k[f](E_f) = 0$ where $\nu_m^k[f]$ is the uniform measure of $G_m^k$ on $f$.
\end{lemma}

\begin{proof}
    Let $\m U = (\v u_1|\ldots|\v u_{m})$ be an element in the Stiefel manifold $\st_n^m$. 
    Define two projection maps $P_{n, k}:\st_n^m\rightarrow G_n^k$ and $P_{n, m}:\st_n^m \rightarrow G_n^m$ by
    \begin{gather*}
        P_{n, k}(\m U) = \spn\set{\v u_1,\ldots,\v u_{k}},\\
        P_{n, m}(\m U) = \spn\set{\v u_1,\ldots,\v u_{m}}.
    \end{gather*}
    Let $\gr_n^m$ be the uniform measure on $\st_n^m$.
    Then the pushforward measure $(P_{n,k})_{\ast}\gr_n^m$ is the uniform measure $\nu_n^k$ up to a constant scalar.
    Denote $\nu_m$ the uniform measure on $O(m)$. 
    Since the map $P_{n, m}$ is a Riemannian submersion, by the co-area formula \cite{chavel2006riemannian},
    \begin{equation*}
        0 = \int \ind{E} d \nu_n^k = \const \int \ind{P_{n,k}^{-1}(E)} d\gr_n^m = \const \int_{G_n^m} d\nu_n^m(f) \int_{O(m)} \ind{P_{n,k}^{-1}(E) \inter P_{n, m}^{-1}(f)} d\nu_m.
    \end{equation*}
    Thus, there is a null set $F \sset G_n^m$ such that for all $f \notin F$, the second integral is 0.
    Fix any $f \notin F$ and $\m V_0 \in \st_n^m$ a basis of $f$.
    We isometrically identify the fiber $P_{n, m}^{-1}(f)$ with $O(m)$ through $\iota: P_{n, m}^{-1}(f) \ni \m U \mapsto \m V_0^\top \m U \in O(m)$.
    Define the map $\phi:O(m) \rightarrow G_m^k[f]$ by 
    \begin{equation*}
        \phi(\m H) = \range\left(\m V_0 \m H \begin{pmatrix}\m I_{k}\\\m 0\end{pmatrix}\right).
    \end{equation*}
    Note the pushforward measure $\phi_{\ast} \nu_m$ is the uniform measure $\nu_m^k[f]$ up to a constant scalar.
    Now, 
    \begin{equation*}
        0
        =
        \int_{O(m)} \iota(P_{n,k}^{-1}(E) \inter P_{n, m}^{-1}(f)) d\nu_m
        =
        \const \int \phi\circ \iota(P_{n,k}^{-1}(E) \inter P_{n, m}^{-1}(f)) d \nu_m^k[f]
        =
        \int \ind{E_f} d \nu_m^k[f].
    \end{equation*}
    The last step follows since $e\in E_f$ if and only if one can find a basis of $e$ being the first $k$ column of an element in $P_{n,k}^{-1}(E) \inter P_{n, m}^{-1}(f)$. 
    Thus, we have found the desired null set $F$.
\end{proof}

We now use the tool \cref{lem:gr_integral} to prove the lemmas in \cref{sec:proof}.

\measregset*

\begin{proof}
    For any $n' \geq m$, denote $E_{n'}$ the null set in $\mu_N^{N-n'}$ consists of all planes at which the conditions in \cref{def:reg_set} fails for $V$.
    By definition, 
    \begin{equation*}
        0
        =
        \mu_N^{N-n'}(E_{n'}) 
        = 
        \int d\nu_N^{N-n'}(P) \int \ind{(\v x, P) \in E_{n'}} d\gl_{n'}(\x).
    \end{equation*}
    Thus, for $\nu_N^{N-n'}$-almost all $P$, the set $I_P = \set{\x: (\x, P) \in E_{n'}}$ has $\gl_{n'}$ measure 0.
    Denote by $F$ the null set of $\nu_N^{N-n'}$ such that the above fail to hold.
    By \cref{lem:gr_integral}, there is a set $H \sset G_N^{N-m}$ of full measure, such that for $Q \in H$, the set $F_Q^{n'} := \{P\in G_N^{N-n'}: P\sset Q, P\in F\}$ has $\nu_{N-m}^{N-n'}[Q]$ measure 0.
    Finally, fix $Q\in H$,
    and for an affine vector $\y$ of $Q$,
    denote $J_{\y, P} = \{\v z \in Q/P: (\y + \v z) \in I_P\}$.
    For $P \notin F_Q^{n'}$, by the Fubini theorem,
    \begin{equation*}
        0 = \gl_{n'}(I_P) = \int d\gl_m(\v y) \int \ind{\v y + \v z \in I_P} d\gl_{n'-m}(\v z). 
    \end{equation*}
    Thus, for almost all $\v y$, set $J_{\v y, P}$ has a $\gl_{n' - m}$ measure 0 for all $P \notin F_Q^{n'}$.
    Denote this set of good $\v y$ vectors by $R_Q$.
    Denote $T = \set{(\v y, Q): Q \in H, \v y \in R_Q}$, and for each $L = (\v y, Q) \in T$, denote $S_L = \{(\v y + \v z, P): P \notin F_Q^{n'}, \v z \notin J_{\v y, P}\}$.
    We deduce from above that $\mu_N^{N-m}(T^c) = 0$ and $\mu_{N-m}^{N - n'}[L]((S_L)^c) = 0$.
    Since by definition, $S_L \inter E_{n'} = \eset$, all planes in $S_L$ are generic to $V$.
    This means for all $L' \in S_L$ of codimension $r > n - m$ in $L$, it has a codimension of $m + r > n$ in $\R^N$, so $L' \inter (L\inter V) = L'\inter V = \emptyset$.
    And for all $L' \in S_L$ of codimension $r > n - m$ in $L$, it has a codimension of $m + r \leq n$ in $R^N$, so $L'\inter (L \inter V) = L'\inter V$ has at most $K$ connected components.
    The assertion is thus proved.
\end{proof}

\kzeroregset*

\begin{proof}
    Let $V\sset \R^N$ be a $(K, 0)$ regular set.
    Towards a contradiction, assume there is a path connected component $A$ of $V$ such that $\diam(A) > 0$.
    For such $A$, there exist some $\gee,\gd > 0$ such that there is a set $E \sset G_{N}^1$ of measure $\gee$ so that the projection of $A$ onto any element $P\in E$ has a diameter of at least $\gd$.
    Denote the unit tangent vector of $P$ by $\hat{\v p}$.
    By continuity, for each $P\in E$, there is an interval $F_P\in \R$ of length at least $\gd$ such that $(u\hat{\v p}, P^\perp) \inter A \neq \emptyset$ for all $u \in F_P$. 
    The collection of such $(u\hat{\v p}, P^\perp)$ has a measure of at least $\gee\gd > 0$ in $\mu_{N}^{N-1}$.
    This violates the assumption that $V$ is $(K, 0)$ regular.
\end{proof}

\section{Additional Proofs in \cref{sec:regsets}} \label{app:regset}

Throughout this section, let $L = (\u, P)$ be an affine subspace of $\mathbb{R}^N$ we use to intersect with the set to be proved to be regular. 
We denote $m$ the codimension of $L$.
Let $\v e_1,\ldots,\v e_m$ be an orthonormal basis of $P^\perp$. 
For a subset $S$ of Euclidean space, let $CC(S)$ denote the number of path connected components of $S$ with respect to the Euclidean topology.

\regvariety*

\begin{proof}
    We will only prove (c) since the rest will be automatic once we prove \cref{lem:regsemi} and take $b$ therein to be respectively 0 and 1.

    For (c), consider $S = \{(\y, w) :  \| \y \|^2 w = 1, \, \y \in V, \, \pi(\y) \in L\}$ in $\R^N \times \R$.
    The equations for $S$ read
    \begin{gather*}
        \ip{\y}{\v e_i}^2 = \ip{\v u}{\v e_i}^2 \norm{\y}^2,\ \ \ i = 1,\ldots,m;\\
        f_j(\y) = 0,\ \ \ \text{for all } j;\\
        \norm{\y}^2 w = 1.
    \end{gather*}
    The polynomials all have degree at most $d' = \max\set{3, d}$.
    Thus, for all $d \geq 2$, $CC(S) \leq d' (2d'-1)^N \leq (2d+1)^{N+1}$. 
    For $m > n$, we have $S = \eset$ for Zariski generic $L$ again by \cite[Theorem~1.1]{dalbec1995introduction}.
    Since $\pi$ is continuous on $S$ with image $\pi(V) \cap L$,
    we conclude $\pi(V)$ is $((2d+1)^{N+1}, n)$ regular.  As in the proof of \cref{lem:polynom_regularity}, if $V$ is a cone then we can replace $n$ by $n-1$ in the second regularity parameter.
\end{proof}

\regrat*

\begin{proof}
    Suppose $f(\x) = (\frac{p_1(\x)}{q_1(\x)}, \ldots, \frac{p_N(\x)}{q_N(\x)})$ for polynomials $p_i$ and $q_i$  of degrees at most $d$ (where $f$ is defined).  
    Let $S = \{(\x, w) : w \prod_{i=1}^N q_i(\x) = 1, \, f(\x) \in L\}$ in $\R^n \times \R$.
    Then $S$ is the solution set to
    \begin{gather*}
            \ip{g(\x)}{\v e_i} = \ip{\v u}{\v e_i} \prod\nolimits_{i=1}^N q_i(\x) ,\ \ \ i = 1,\ldots,m;\\
            w \prod\nolimits_{i=1}^N q_{i}(\x) = 1,
        \end{gather*} 
    where $g(\x)$ is obtained by clearing denominators in $f(\x)$, that is, 
    \begin{equation*}
        g(\x) := \Big{(}p_1(\x) \prod_{i \neq 1} q_i(\x), \,\, \ldots \,\,, p_N(\x) \prod_{i \neq N} q_i(\x)\Big{)}.
    \end{equation*}
    Here the polynomial equations have  degree at most $dN + 1$.
    Thus POMT implies $CC(S) \leq (dN+1)(2dN+1)^{n} \leq (2dN+1)^{n+1}$. 
    Again, $S = \emptyset $ Zariski generically if $m > n$, because $\im(f)$ is contained in an algebraic variety of dimension at most $n$.
    Since $(\x, w) \mapsto f(\x)$ is continuous on $S$ with image $\im(f) \cap L$, it follows $\im(f)$ is $((2dN)^{n+1}, n)$ regular. 
\end{proof}

\regsemi*

\begin{proof}
    First, note that $U \inter L$ is a semialgebraic set with $b$ inequality constraints.
    As shown in the proof of \cite[Theorem 4.8]{yomdin2004tame}, to get an upper bound on $CC(U\inter L)$, we may assume that all $b$ inequalities are of form $q_\ell(\y) \leq 0$ by passing to a subset of $U\inter L$ with at least as many components.
    Now the set $S = \{(\y, \v w): q_\ell(\y) \leq 0~\forall \ell\in [b],~\y\in U\inter L\}$ in $\R^N\times \R^b$ is a semialgebraic set defined by
    \begin{gather*}
        \ip{\y}{\v e_i} = \ip{\v u}{\v e_i},\ \ \ i = 1,\ldots,m;\\
        p_k(\y) = 0,\ \ \ \text{for all } k;\\
        q_\ell(\y) \leq 0\ \ \ \ell = 1,\ldots,b.
    \end{gather*}
    Using \cite[Theorem 7.38]{basu2006algorithms}, we can bound $CC(S)\leq \dim(V') \cdot 7^b (2d)^N$ uniformly over $m$, where $V'$ is the variety defined by polynomials $p_k$, and $\dim(V') \leq N$.
    On the other hand, we can introduce $b$ dummy variables and replace the $q_\ell$ inequalities with
    \begin{equation*}
        q_\ell(\y) + \go_\ell^2 = 0.
    \end{equation*}
    Then POMT gives $CC(S) \leq (2d)^{N + b}$.
    Finally, the results in \cite{gabrielov2009approximation} states that $CC(S) \leq (2d)^N \cdot (\const b^2)^N$.
    The constant $\const$ to the best of our knowledge has no clear specification.
    As the continuous map $(\v y, \v w) \mapsto \v y$ sends $S$ onto $U\inter L$. 
    Thus, $CC(U\inter L) \leq CC(S) \leq (2d)^N\min\set{7^b, (2d)^b, (\const b^2)^N}$.
    Again, when $m > n$ and $L$ is Zariski generic, $S = \eset$ because a semialgebraic set of dimension at most $n$ is contained in a variety of dimension at most $n$. 
    Hence the proof is complete. 
\end{proof}

\section{Additional Proofs in \cref{sec:cp}} \label{app:cp}

\grasstostief*

\begin{proof}
    Following the classic theory of principal angles, we can find an orthonormal basis $\m U = (\v u_1|\ldots|\v u_k)$ of $E$ and $\m V = (\v v_1|\ldots|\v v_k)$ of $F$ such that
    \begin{equation*}
        \m P_E\m P_F = \m U \m D \m V^\top,
    \end{equation*}
    where $\m D$ is a nonnegative diagonal matrix satisfying 
    $\U^\top \V = \D$ (in particular $d_i = \ip{\v u_i}{\v v_i} \in [0, 1]$).
    Following the arguments in \cite{pajor1998metric}, for $d_i < 1$, define
    \begin{equation*}
        \tilde{\v u}_i = \frac{-d_i \v u_i + \v v_i}{\sqrt{1-d_i^2}},~\text{and}~
        \tilde{\v v}_i = \frac{-\v u_i + d_i \v v_i}{\sqrt{1 - d_i^2}}.
    \end{equation*}
    Then the rank at most $2k$ matrix $\m P_E - \m P_F$ has the SVD
    \begin{equation*}
        \m P_E - \m P_F = \sum_{i:~d_i < 1} \sqrt{1-d_i^2}\left(\tilde{\v u}_i \v v_i^\top + \v u_i \tilde{\v v}_i^\top\right).
    \end{equation*}
    Without loss of generality, let us assume $\|\P_E - \P_F\| = \eps$.
    The above then implies that $\eps \leq 1$, and
    consequently the norm condition implies $\min_i d_i \geq \sqrt{1 - \gee^2}$.

    Now for any basis $\m Q_E$ of $E$, let $\m M$ be the $k\times k$ orthogonal matrix such that $\m Q_E = \m U\m M$.
    The candidate basis for $F$ is $\m Q_F = \m V\m M$.
    Indeed, with this choice of $\m Q_F$, we have
    \begin{align*}
        \norm{\m Q_E - \m Q_F}_2
        &= \norm{\m U - \m V}_2 
            = \|(\m U - \m V)^\top (\m U - \m V)\|_2^{1/2} 
            = \sqrt{2} \norm{\m I - \m D}_2^{1/2}\\
        &\leq \sqrt{2}\left(1 - \sqrt{1 - \gee^2}\right)^{1/2}
            \leq \sqrt{2} \gee.
    \end{align*}
    Thus the proof is complete.
\end{proof}

\section{Additional Proofs in \cref{sec:sketching}} \label{app:sketching}

\begin{lemma}
    \label{lem:bnd_sqrt_ln}
    For $0 \leq a \leq  b \leq c$, we have
    \begin{equation*}
        \int_a^b \sqrt{\log(c/\gee)} \,d\gee \leq b\sqrt{\pi} + \gee\sqrt{\log(c/\gee)}\Big\rvert_a^b.
    \end{equation*}
\end{lemma}

\begin{proof}
    Integration by parts gives
    \begin{equation*}
        \int_a^b \sqrt{\log(c/\gee)} \,d\gee
        = 
        \gee\sqrt{\log(c/\gee)}\Big\rvert_a^b
            + \frac{1}{2}c\sqrt{\pi}\left .\erf\left(\sqrt{\log\frac{c}{\gee}}\right)\right\vert_b^a,
    \end{equation*}
    where $\erf$ is the error function $\erf(x) = \int_0^x \frac{2}{\sqrt{\pi}}e^{-y^2}\,dy$.
    So the error function term is upper bounded by
    \begin{equation*}
        \left .\erf\left(\sqrt{\log\frac{c}{\gee}}\right)\right\vert_b^{0}
        =
        \pr\left(|\cN(0, 1)| \geq \sqrt{2}\sqrt{\log(c/b)}\right)
        \leq
        2e^{-\log(c/b)}. 
    \end{equation*}
    Plugging into the exact integration, we get the claimed result.
\end{proof}

The following result is classic, see for example \cite{vershynin2018high}, but it is more convenient in our context.

\begin{lemma}\label{lem:subG_norm}
    Let $\S \in \R^{m \times M}$ be a matrix in $\subG(\ga)$.
    There exist some absolute constants $c_1, c_2 > 0$, so that for any $\gd \in (0, 1)$ and $u > \sqrt{c_2/c_1}$, 
    \begin{equation*}
        \pr\left(\norm{\S} \geq \ga u\right) \leq \gd
    \end{equation*}
    provided that
    \begin{equation*}
        m \geq (c_1u^2 - c_2)^{-1}(c_2 M + \log(1/\gd)).
    \end{equation*}
\end{lemma}

\begin{proof}
    We follow the proof in \cite[Theorem 4.4.5]{vershynin2018high}. 
    Using a standard net argument,
    we know that for $t \geq 0$,
    \begin{equation*}
        \norm{\S} = \max_{\x \in \cB^m, \y \in \cB^M} \ip{\x}{\S \y} \geq t
    \end{equation*}
    happens with probability at most
    \begin{equation*}
        p(t) := 9^{m + M}\cdot 2 \exp\big(- c_1 m (t/\ga)^2\big)
        \leq
        \exp(c_2(M + m) - c_1 m (t/\ga)^2). 
    \end{equation*}
    Taking $t = \ga u$ with $u^2 > c_2/c_1$, we see that $p(\ga u) \leq \gd$ as long as $m$ satisfies the condition in the statement.
\end{proof}

\section{Additional Proofs in \cref{sec:gen_err_rat}}

In this section, we list the lemmas used to prove \cref{thm:covering_ratnn} as well as \cref{cor:gen_err_rational} together with their proofs.
We denote $\gS_nR^\ga_\gb$ denotes a sum of $n$ rational functions in $R^\ga_\gb$.

\begin{lemma}
    \label{lem:degree_single_layer}
    Suppose $\m X^*_{i}(\gL) \in R^\theta_\theta$ and the rational function $\gr_{i+1} \in R^s_s$ are fixed.
    Then for $\ratnn$ \eqref{eq:def_ratnn}
    \begin{equation} \label{eq:ratnn_notrain_deg_layer}
        \m X^*_{i+1}(\gL) \in R^{s(1 + d_i\theta)}_{s(1 + d_i\theta)};
    \end{equation}
    and for $\ratcnn$ \eqref{eq:def_ratcnn},
    \begin{equation} \label{eq:ratcnn_notrain_deg_layer}
        \m X^*_{i+1}(\gL) \in R^{s(1 + c_{i}k^2\theta)}_{s(1 + c_{i}k^2\theta)}.
    \end{equation}
    If instead the rational activation is trainable with variable coefficients, then
    for $\ratnn$ \eqref{eq:def_ratnn}
    \begin{equation}\label{eq:ratnn_train_deg_layer}
        \m X^*_{i+1}(\gL) \in R^{s(1 + d_i\theta) + 1}_{s(1 + d_i\theta) + 1};
    \end{equation}
    and for $\ratcnn$ \eqref{eq:def_ratcnn},
    \begin{equation}\label{eq:ratcnn_train_deg_layer}
        \m X^*_{i+1}(\gL) \in R^{s(1 + c_{i}k^2\theta) + 1}_{s(1 + c_{i}k^2\theta) + 1}.
    \end{equation}
\end{lemma}

\begin{proof}
    First we prove the non-trainable case.
    For rational neural networks, $\m A_{i+1} \m X^*_{i}(\gL) + \v b_{i+1} \in \gS_{d_i} R^{\theta+1}_\theta + R^1_0 \sset R^{d_i\theta+1}_{d_i\theta} + R^1_0 \sset R^{d_i\theta + 1}_{d_i\theta}$.
    Then $\m X^*_{i+1}(\gL) \in R^s_s \circ R^{d_i\theta + 1}_{d_i\theta} \sset R^{s (d_i\theta + 1)}_{s (d_i\theta + 1)}$ proving the first assertion.

    For convolutional case, denote $\m K_{pq}$ for $p \in [c_{i+1}]$ and $q \in [c_{i}]$ the kernel in $\R^{k\times k}$.
    Denote $\m X^*_i(\gL)_q \in \R^{d_i}$ the $q$th channel of the input $\m X^*_i(\gL) \in \R^{c_i \times d_i}$.
    Then $\m K_{pq} \ast \m X^*_i(\gL)_q \in \sum_{k^2} R^{\theta +1}_\theta \sset R^{k^2\theta+1}_{k^2\theta}$.
    Summing over $q \in [c_{i}]$, we have
    $\t K_i \m X^*_i(\gL) \in \sum_{c_i} R^{k^2\theta+1}_{k^2\theta} \sset R^{c_ik^2\theta + 1}_{c_ik^2\theta}$.
    Finally $\m X^*_{i+1}(\gL) \in R^s_s\circ (R^{c_ik^2\theta + 1}_{c_ik^2\theta} + R^1_0) \sset R^{s(1 + c_{i}k^2\theta)}_{s(1 + c_{i}k^2\theta)}$ proving assertion 2.

    For cases where the activation is trainable, the coefficients in $\gr_{i+1}$ simply raise the degree of both numerator and denominator by 1. Thus the next two assertions follow directly.
\end{proof}

\begin{lemma} \label{lem:degree_ratnn}
    Suppose the activations in $\ratnn$ or $\ratcnn$ are in $R^s_s$ and $k \geq 2$, $d_i \geq 2$ for all $i \leq L-1$.
    If the activations are not trainable, then for $\ratnn$, 
    \begin{equation*}
        \m X^*(\gL) \in R^{d^\sharp}_{d^\sharp},~~d^\sharp = 2d_1\ldots d_{L-1} s^L;
    \end{equation*}
    and for $\ratcnn$,
    \begin{equation*}
        \m X^*(\gL) \in R^{d^*}_{d^*},~~d^* = 2c_1\ldots c_{L-1} s^Lk^{2(L-1)}.
    \end{equation*}
    If the activations are trainable, then respectively
    \begin{equation*}
        d^\sharp = 2(1+s^{-1})d_1\ldots d_{L-1} s^L,~~d^* = 2(1+s^{-1})c_1\ldots c_{L-1} s^Lk^{2(L-1)}.
    \end{equation*}
\end{lemma}

\begin{proof}
    For non-trainable $\ratnn$ case, suppose $\m X^*_i(\gL) \in R^{\gD_i}_{\gD_i}$, and repeatedly apply \eqref{eq:ratnn_notrain_deg_layer}, we get 
    \begin{equation*}
        \gD_1 = s,~~\gD_{i+1} = s(1 + d_i\gD_i).
    \end{equation*}
    Thus $\gD_L = s + s^2d_{L-1} + s^3d_{L-1}d_{L-2} +\ldots + s^Ld_{L-1}\ldots d_{1}$. Since $s \geq 1$ and $d_i \geq 2$ for $i \leq L-1$, the ratios between adjacent terms are at least 2. Thus $\gD_L \leq d^\sharp$ as claimed.
    By comparing \eqref{eq:ratcnn_notrain_deg_layer} and \eqref{eq:ratnn_notrain_deg_layer}, the convolution map $\t K_i$ can be viewed as a linear map of dimension $\R^{c_{i} \times c_{i-1}k^2}$. Thus, the formula for $d^*$ follows from the same argument.
    Finally, for $\ratnn$ with trainable activations, the recursion relation becomes
    \begin{equation*}
        \gD_1 = s+1,~~\gD_{i+1} = s(1 + d_i\gD_i) + 1.
    \end{equation*}
    So $\gD_L = (s+1)(1 + sd_{L-1} + s^2d_{L-1}d_{L-2} + \ldots + s^{L-1}d_{L-1}\ldots d_1) \leq (1 + s^{-1}) \cdot 2d_1\ldots d_{L-1} s^L$. The conclusion for $\ratcnn$ follows similarly.
\end{proof}

\section{Additional Proofs in \cref{sec:gen_err_common}} \label{app:gen_err_common}

In this section we give the full proof of \cref{thm:gen_err_common_full}.
Notations and detailed setup are defined in \cref{sec:gen_err_common}.
Our main strategy is to approximate the ReLU function at each layer using rational activations, and then apply our result on rational neural networks to control the covering number of the image $\cF(\m X)$. 
This is followed by the standard Dudley integral argument to control the Rademacher complexity.
We start by results on rational approximations.
The following result is essentially the same as Lemma 1 in \cite{boulle2020rational}, modified to adapt to our context.

\begin{lemma} \label{lem:rat_approx_relu}
    (Lemma 1 in \cite{boulle2020rational}, modified)
    For any $t > 0$ and $\gee > 0$, there exists $\gr \in R^{1 + \gD}_{\gD}$, where $\gD \leq \frac{6}{\pi^2} \log_+^2 (4 t/\eps)$, such that $\norm{\gr - \relu}_{L^\infty(-t, t)} \leq \eps$, and $\im(\gr) \sset [0, t]$ on $[-t, t]$.
\end{lemma}

\begin{proof}
    The assertion is obviously true if $\gee \geq t$, where we can take $\gr$ to be a constant function.
    Consider $\gee \in (0, t)$.
    By Lemma 1 in \cite{boulle2020rational}, for any $\gd \in (0, 1)$, there exists a rational map $\gr_0$ being a composition of $k$ maps in $R^3_2$ such that $|\gr_0(x)|\leq 1$ on $[-1, 1]$ and $\norm{\gr_0 - \relu}_{L^\infty(-1, 1)} \leq \gd$, where
    \begin{equation*}
        k = \ceil{\frac{\log(2/\pi^2) + 2 \log\log(4/\gd)}{\log 3}}
        \leq \frac{\log(6/\pi^2) + 2 \log\log(4/\gd)}{\log 3}.
    \end{equation*}
    Therefore, the degree of $\gr_0$ is bounded by
    \begin{equation*}
        3^k \leq \frac{6}{\pi^2} \log^2 (4 /\gd).
    \end{equation*}
    Now set $\gd = \eps / t$, and define $\gr(x) = t\gr_0(x/t)$, then $|\gr(x)| \leq t$ on $[-t, t]$ and for any $x \in [-t, t]$, we have
    \begin{equation*}
        |\gr(x) - \relu(x)| = t |\gr_0(x/t) - \relu(x/t)| \leq \eps. 
    \end{equation*}
    This finishes the proof.
\end{proof}

When the $\relu$ activation is approximated be a rational map at layer $i$, this approximation error will propagate to later layers.
Next we control this propagation error.
We first fix some notations.
For $i \in [L]$, we denote the bound on the input to layer $i$ by $\l_{i-1}$. 
That is, for network $f$,
$\norm{f_{i-1}(\m X)}_{\max} \leq \l_{i-1}$.
By convention, since the data is bounded as $\norm{\m X}_{\max} \leq 1$, we define $\l_0 = 1$.
As illustrated previously, we replace ReLU activations in $f$ by rational approximations. 
We denote the rational approximation in layer $i$ as $\gr_i$.
we will use \cref{lem:rat_approx_relu} to construct these $\gr_i$, and thus we are guaranteed that $|\gr_i(x)|\leq \go_i\l_{i-1}$ on $x \in [-\go_i\l_{i-1}, \go_i\l_{i-1}]$.
Hence we have $\l_i \leq \prod_{j < i}\go_j$.
We denote the resulting rational network by $\hat{f}$.
The following lemma controls the propagation error.

\begin{lemma} \label{lem:propagation}
    Let everything be defined as above.
    If $\|f_{i-1}(\x) - \hat{f}_{i-1}(\x)\|_{\infty} \leq \gd_{i-1}$ for all $\norm{\x}_{\infty} \leq 1$ and $\norm{\gr_{i} - \relu}_{\infty} \leq \gee_i$ on $[-\go_i\l_{i-1}, \go_i\l_{i-1}]$, 
    then $\|f_i(\x) - \hat{f}_i(\x)\|_{\infty} \leq \gee_i + \go_i\gd_{i-1}$ for all $\norm{\x}_{\infty} \leq 1$.
\end{lemma}

\begin{proof}
    This result is standard, and similar results may be found in many other works. We include one here to make the narration self-contained.
    The calculation is
    \begin{align*}
        \gd_i 
        :=& 
        \sup_{\norm{\x}_{\infty} \leq 1} \|f_i(\x) - \hat{f}_i(\x)\|_{\infty}\\
        \leq&
        \sup_{\norm{\x}_{\infty} \leq 1} \|\gr_i(\m A_i \hat{f}_{i-1}(\x) + \v b_i) - \relu(\m A_i \hat{f}_{i-1}(\x) + \v b_i)\|_{\infty}\\
        &
        ~~~~+\sup_{\norm{\x}_{\infty} \leq 1} \|\relu(\m A_i \hat{f}_{i-1}(\x) + \v b_i) - \relu(\m A_i f_{i-1}(\x) + \v b_i)\|_{\infty}\\
        \leq&
        \eps_i + \go_i \gd_{i-1}.
    \end{align*}
    In the second line we used $\|\m A_i \hat{f}_{i-1}(\x) + \v b_i\|_{\infty} \leq \go_i\l_{i-1}$.
    The result is thus proved.
\end{proof}

In the next lemma, we control the degree of the coordinate functions of rational network $\hat{f}$, which is similar to \cref{lem:degree_ratnn}.

\begin{lemma} \label{lem:degree_rat_approx}
    Suppose rational activations $\gr_i$ at layer $i$ is in $R^{1+\gD_i}_{\gD_i}$ as suggested in \cref{lem:rat_approx_relu}.
    Then $\hat{f}_L \in R^{\theta_L}_{\theta_L}$ where $\theta_L = \prod_{i = 1}^L(1 + \gD_i) \prod_{j = 1}^{L-1}(1 + d_j)$.
\end{lemma}

\begin{proof}
    Since $\gr_1 \in R^{1 + \gD_1}_{\gD_1}$, we have $\hat{f}_1(\x) = \gr_1(\m A_1 \x + \v b_1) \in R^{1+\gD_1}_{\gD_1}$ as a function in $\gL$.
    Let $\theta_1 = 1 + \gD_1$. 
    Then
    $\hat{f}_1 \in R^{\theta_1}_{\theta_1}$. 
    Now $\hat{f}_2(\x) = \gr_2(\m A_2\hat{f}_1(\v x) + \v b_1) \in R^{1+ \gD_2}_{\gD_2} \circ (\gS_{d_1} R^{\theta_1}_{\theta_1} + R^1_0) \sset R^{(1 + d_1\theta_1)(1+\gD_2)}_{(1 + d_1\theta_1)(1+\gD_2) - 1} \in R^{\theta_2}_{\theta_2}$, where $\theta_2 = (1 + d_1\theta_1)(1 + \gD_2)$.
    Inductively repeating this argument, we have  $\hat{f}_L \in R^{\theta_L}_{\theta_L}$, where $\theta_{i+1} = (1 + d_i \theta_i)(1 + \gD_{i+1})$ with $\theta_0 = 0$.
    Since $1 + d_i\theta_i \leq (1 + d_i)\theta_i$,
    we have $\theta_L \leq \prod_{i = 1}^L (1 + \gD_i) \prod_{j = 1}^{L-1}(1 + d_j)$ as claimed.
\end{proof}

Now we combine everything together to prove \cref{thm:gen_err_common_full}. 
We will need to choose the quality of approximation at each layer.

\generrReLU*

\begin{proof}
    We will only prove the first bound. The simplified bound is an easy check.
    We continue to use notations $\gee_i$ and $\gd_i$ as defined in \cref{lem:propagation}.
    We set the approximation error of the first layer to be $|\gr_1(x) - \relu(x)| \leq \gee$ for all $x \in (-\go_1, \go_1)$. The value of $\gee$ will be chosen later.
    By definition, $\gd_1 = \gee$.
    For later layers, we choose $\gee_i = \gd_{i-1}$.
    From \cref{lem:propagation}, we have
    \begin{equation*}
        \gd_i \leq (1 + \go_i)\gd_{i-1},~i\geq 2,~\text{and}~\gd_L \leq \gee\cdot \prod_{i = 2}^L (1 + \go_i).
    \end{equation*}
    By \cref{lem:rat_approx_relu}, the required degree of each $\gr_i$ (with respect to $\Lambda$) is then bounded as
    \begin{equation*}
        \gr_i \in R^{1 + \gD_i}_{\gD_i},
        ~\text{where}~
        \gD_i \leq \frac{6}{\pi^2}\log_+^2(4 \go_i \l_{i-1} / \gd_{i-1}),~i\geq 2,
        ~\text{and}~
        \gD_1 \leq \frac{6}{\pi^2}\log_+^2(4 \go_1 / \eps).
    \end{equation*}
    By \cref{lem:degree_rat_approx}, the final approximation network is a rational map $\hat{f} \in R^{d_f}_{d_f}$ (with respect to $\Lambda$), where
    \begin{equation*}
        d_f = \prod_{i = 1}^L(1 + \gD_i) \prod_{j = 1}^{L-1}(1 + d_j).
    \end{equation*}  
    These rational networks approximating networks in $\cF = \relu(L, \v d, \v \go)$ form a new hypothesis class, denoted as $\widehat{\cF}$.
    For any $f \in \cF$, one can find $\hat{f} \in \widehat{\cF}$ so that $\|f(\x) - \hat{f}(\x)\|_{\infty} \leq \gee\prod_{i = 2}^L(1 + \go_i)$, for all $\|\x\|_{\infty}\leq 1$.
    Therefore,
    \begin{align*}
        \fR_S(\cF;\l) 
        &=
        \E_{\v \gs}\sup_{f\in \cF}\frac{1}{n}\sum_{i = 1}^n \l(f(\v x_i), y_i) \cdot \gs_i\\
        &=
        \E_{\v \gs}\frac{1}{n}\sum_{i = 1}^n \l(f_{\v \gs}(\v x_i), y_i) \cdot \gs_i\\
        &=
        \E_{\v \gs}\frac{1}{n}\sum_{i = 1}^n \Big(\l(f_{\v \gs}(\v x_i), y_i) - \l(\hat{f}_{\v \gs}(\v x_i), y_i)\Big)\cdot \gs_i + \l(\hat{f}_{\v \gs}(\v x_i), y_i) \cdot \gs_i\\
        &\leq 
        \E_{\v \gs}\norm{\l}_{\lip} \|f_{\v \gs} - \hat{f}_{\v \gs}\|_{\infty}
            +
        \fR_S(\widehat{\cF};\l)\\
        & \leq 
        \norm{\l}_{\lip}\gee \cdot \prod_{i = 2}^L(1 + \go_i) + \fR_S(\widehat{\cF};\l).
    \end{align*}
    In the second line $f_{\v \gs}$ is the maximizer of the inner product given $\v \gs$, which exists since the parameter space is compact.
    In the third line $\hat{f}_{\v \gs}$ is some function from $\widehat{\cF}$ that achieves the approximation error discussed before.
    With $\ga = \min\set{H, \norm{\ell}_{\lip}\sqrt{d_L}\prod_{i = 1}^L \go_i}$,
    following the proof of \cref{cor:gen_err_rational}, we have
    \begin{equation*}
        \fR_S(\widehat{\cF};\l) 
        \leq 
        \const \ga\sqrt{\frac{N}{n}}\left(\sqrt{\pi} + \sqrt{\log\frac{W\norm{\l}_{\lip}}{\ga\sqrt{n}}}\right),
    \end{equation*}
    where $N$ is the number of parameters in the network, and $W = N \left(\prod_{i = 1}^L \go_i\right) (nd_{L})^{5/2} d_f$.
    Simplifying this expression, we get
    \begin{equation*}
        \fR_S(\widehat{\cF};\l) \leq \const n^{-1/2} \ga \sqrt{N\left(\log(n) + \log\left(\frac{\norm{\l}_{\lip}}{\ga}\right) + \sum_{i = 1}^L \log(d_i\go_i(1 + \gD_i))\right)}.
    \end{equation*}
    We are left to choose $\eps$ and plug in the bound for $\gD_i$.
    We choose $\eps$ small so that $\norm{\l}_{\lip} \gee \prod_{i = 2}^L(1 + \go_i) = n^{-1/2} \ga$.
    Then for $i \geq 2$
    \begin{equation*}
        \frac{\go_i\l_{i-1}}{\gd_{i-1}} 
        \leq
        \frac{\prod_{j=1}^i\go_j}{\prod_{j = 2}^{i-1} (1 + \go_j)} \cdot \frac{\norm{\l}_{\lip}}{\ga}\sqrt{n}\prod_{k = 2}^L(1 + \go_k)
        \leq
        \go_i\frac{\norm{\l}_{\lip}}{\ga}\sqrt{n} \prod_{j = 1}^L(1 + \go_j).
    \end{equation*}
    Since $\go_i \geq 2$, $\log(1 + \gD_i) = \cO\Big(1 + \log_+\log_+({\ga}^{-1}\norm{\l}_{\lip}\cdot \sqrt{n}\cdot \prod_{i = 1}^L \go_i)\Big)$ for $i \geq 2$.
    One can verify that this is also true for $\gD_1$.
    Replacing $\gD_i$ by the bound above leads to the claimed result.
\end{proof}

\begin{remark}
    There are recent results on rational approximations to other  activation functions, such as $\tanh$ \cite{boulle2020rational}. 
    Therefore \cref{thm:general_cover_number} is applicable to networks with other activations too.
\end{remark}

\end{document}